\documentclass{amsart}
\usepackage{amsmath,amscd,amsthm,amsfonts,a4wide,amssymb}
\usepackage[all,cmtip]{xy}
\usepackage{fancyhdr}
\theoremstyle{plain}
\newtheorem{theorem}                 {Theorem}      [section]
\newtheorem{proposition}  [theorem]  {Proposition}
\newtheorem{corollary}    [theorem]  {Corollary}

\newtheorem{lemma}        [theorem]  {Lemma}

\theoremstyle{definition}
\newtheorem{example}      [theorem]  {Example}
\newtheorem{remark}       [theorem]  {Remark}
\newtheorem{definition}   [theorem]  {Definition}

\newcommand{\be}[1]{\begin{equation}\label{#1}}
   \newcommand{\ee}{\end{equation}}

\numberwithin{equation}{section}

\DeclareMathOperator{\ad}{ad}

\DeclareMathOperator{\spa}{span}

\DeclareMathOperator{\rank}{rank}

\newcommand{\wt}{\widetilde}

\newcommand{\cc}{\mathbb C}  %Complexification superscript
\newcommand{\rr}{\mathbb R}  %Reality superscript

\newcommand{\T}{\mathrm T} %Usual transpose
\newcommand{\TT}{\mathfrak T} %second transpose
\newcommand{\A}{\mathfrak A} %set of matrices 
\newcommand{\aaa}{\mathfrak a} %BuGu $\u_o^{\xi}$
\newcommand{\pa}{\partial}
\newcommand{\dbar}{\ov{\partial}}
\newcommand{\zbar}{\ov{z}}
\newcommand{\half}{\tfrac{1}{2}}

\newcommand{\ibar}{\bar{i}}
\newcommand{\Ibar}{\bar{I}}
\newcommand{\jbar}{\bar{j}}

\newcommand{\ellbar}{\bar{\ell}}
\newcommand{\wc}{\wt{c}} % columns minus top \& bottom elements

\def \al{\alpha}
\def \de{\delta}
\def \ga{\gamma}
\def \Ga{\Gamma}
\def \la{\lambda}
\def \La{\Lambda}

\def \N{\mathbb N}
\def \Z{\mathbb Z}
\def \R{\mathbb R}
\def \C{\mathbb C}
\def \H{\mathcal H}

\def \HH{\underline{\mathcal H}}

\def \Gr{Gr}

\def \g{\mathfrak{g}}

\def \t{\mathfrak{t}}

\def\ip#1#2{\langle#1,#2\rangle}

\def \GL{\mathrm{GL}}
\def \SO{\mathrm{SO}}
\def \O{\mathrm{O}}
\renewcommand{\o}{\mathfrak{o}}
\def \U{\mathrm{U}}
\def \SU{\mathrm{SU}}
\def \u{\mathfrak{u}}
\def \Sp{\mathrm{Sp}}

\def \CP{\mathbb{C}P}
\def \RP{\mathbb{R}P}

\def \gl{\mathfrak{gl}}

\def \d{\mathrm{d}}

\newcommand{\marg}[1]{\marginpar{\footnotesize #1}}
\renewcommand{\marg}[1]{} %to comment out marg's

\newcommand{\Cinfty}{{\mathbb C} \cup \infty}	
\newcommand{\CC}{\underline{\mathbb C}} %trivial bundle
\newcommand{\ov}{\overline}
\newcommand{\ul}{\underline}
\newcommand{\alg}{{\rm alg}}

\newcommand{\ii}{\mathrm{i}} %sqrt{-1}

\newcommand{\M}{\mathcal{M}} %space of mero funcs
\newcommand{\Sol}{\mathrm{Sol}} %space of solutions A
\newcommand{\ND}{\text{\rm ND}}

\DeclareMathOperator{\diag}{\mathrm{diag}}

\begin{document}

\bigskip

\title{Harmonic maps into the orthogonal group and null curves}

\author{Maria Jo\~ao Ferreira}
\author{Bruno Ascenso Sim\~oes} 
\author{John C. Wood}

\keywords{harmonic map, null curve, Weierstrass representation, non-linear sigma model}

\subjclass[2000]{53C43, 58E20}

\thanks{The authors thank the Universities of Leeds and Lisbon for hospitality during the preparation of this work, and the Funda\c{c}\~ao para a Ci\^encia e Tecnologia, Portugal and London Mathematical Society for partial financial support.} 

\address
{(MJF and BAS) Centro de Matem\'atica e Applica\c{c}\~oes Fundamentais,  Faculdade de Ci\^encias da Universidade de Lisboa, Campo Grande, 1749-016 Lisboa, Portugal}
\email{(MJF) mjferreira@fc.ul.pt;
(BAS) a{\!\_}s{\!\_}bruno@hotmail.com}

\address
{(JCW) School of Mathematics, University of Leeds, Leeds LS2 9JT, Great Britain}
\email{j.c.wood@leeds.ac.uk}

\begin{abstract}  
We find algebraic parametrizations of extended solutions of harmonic maps of finite uniton number from a surface to the orthogonal group
$\mathrm O(n)$ in terms of free holomorphic data which lead to formulae for all such harmonic maps.  Our work reveals an interesting correspondence between certain harmonic maps and the free Weierstrass representation of null curves and minimal surfaces in 3- and 4-space.
\end{abstract}

\maketitle

\section{Introduction}

\emph{Harmonic maps} are smooth maps between Riemannian manifolds which extremize the `Dirichlet' energy integral (see, for example, \cite{eells-lemaire, urakawa-book}).
They include many interesting classes of mappings, including \emph{geodesics}, \emph{minimal submanifolds} and \emph{harmonic functions}.
Harmonic maps from surfaces to Lie groups and their symmetric spaces are of particular interest, as they admit an integrable systems formulation in terms of \emph{extended solutions},
and they constitute the \emph{chiral} or \emph{non-linear $\sigma$-model} of particle physics, see for example \cite{zakrzewski-book}. 

We give an algorithm (Theorem \ref{th:gen}) which determines, inductively, algebraic parametrizations of extended solutions of harmonic maps of finite uniton number from a surface to the orthogonal group
$\O(n)$ in terms of free holomorphic data; this determines all such harmonic maps.
In contrast to previous work, e.g.\ \cite[\S 6]{unitons}, the holomorphic data is \emph{free}. 
 The parametrizations involves no integration: to avoid that, the algorithm replaces the initial choice of data by new data; this gives \emph{global formulae} for the parametrizations. These formulae determine all harmonic maps locally by choosing the free holomorphic data to be meromorphic functions on open subsets of $M$.  There are two important cases where all extended solutions, and so harmonic maps, are determined globally by our formulae:
 
 (i) \emph{$S^1$-invariant extended solutions for harmonic maps into
$\O(n)$.}  These relate to 
harmonic maps which which arise from \emph{twistor constructions}; these have extended solutions which are invariant under the natural $S^1$-action of C.-L. Terng, see \cite[\S 7]{uhlenbeck}.  An early twistor construction was that of 
E.~Calabi who gave \cite{calabi-quelques, calabi-JDG} a construction of all harmonic maps from the $2$-sphere to real projective spaces or spheres in terms of \emph{totally isotropic holomorphic maps}.   We give a correspondence (Theorem \ref{th:11-appl}) between $S^1$-invariant extended solutions for harmonic maps into $\O(n)$ of maximum uniton number and such totally isotropic holomorphic maps, and so, harmonic maps to spheres.  Using our algorithm, we can give totally explicit global formulae for all these objects (Theorem \ref{th:S1-11}). 

(ii) \emph{The case $n \leq 6$.}  In \S \ref{sec:exs}, by modifying our algorithm in some cases (see, for ex \S \ref{subsec:n=6}(c)), we find global formulae for all harmonic maps of finite uniton number and their extended solutions from a surface to $\O(n)$.  Our formulae have the following interesting application:

A \emph{null curve} is a holomorphic (or meromorphic) map from a surface to $\C^n$ whose derivative is null (isotropic).  The real part of a null curve is a \emph{minimal surface} in $\R^n$ and all minimal surfaces are given that way, locally.   As well as the usual Weierstrass representation involving integration, K.~Weierstrass \cite{weierstrass} gave a formula for such null curves in $\C^3$, called the \emph{free Weierstrass representation}; M.~de~Montcheuil \cite{montcheuil} gave a similar formula for $\C^4$, thus giving (locally) all minimal surfaces in $\R^3$ and $\R^4$ without integration.
Our parametrizations for $n=5,6$ lead to correspondences between certain extended solutions for harmonic maps into $\O(n)$ and null curves
(Theorems \ref{th:Weier3} and \ref{th:Weier4}), where the free Weierstrass data appear very simply in a matrix giving the extended solution.

The starting point is the seminal work of K.~Uhlenbeck \cite{uhlenbeck} who, by introducing a \emph{spectral parameter} $\la$, showed that all harmonic maps from a surface to the unitary group
$\U(n)$ can be obtained, locally at least, from certain maps into its loop group $\Omega\U(n)$, namely
the \emph{extended solutions} mentioned above.  If there is an extended solution polynomial in $\la$, the harmonic map is said to be of \emph{finite uniton number}; all harmonic maps from a compact Riemann surface with a globally defined extended solution, and so all harmonic maps from the $2$-sphere, are of finite uniton number.  Further, Uhlenbeck gave a factorization of a polynomial extended solution into certain linear factors called
\emph{unitons}.  Using the Grassmannian model of the loop group, G.~Segal \cite{segal} showed how to
represent an extended solution by a subbundle $W$ of a trivial bundle with 
fibre a Hilbert space, and showed how to find uniton factorizations from a certain 
natural filtration of $W$. This was put into a general framework in \cite
{unitons}, which led to formulae for uniton factorizations including those of
\cite{dai-terng,ferreira-simoes-wood} (which had been found by different methods).  The minimum number of unitons needed to obtain a given harmonic map is called its \emph{uniton number}.

In \cite{burstall-guest}, a different approach was taken by F.E.~Burstall and M.A.~Guest using a finer classification than that given by uniton number based on a Bruhat decomposition of the algebraic loop group.  This reduced the problem of finding harmonic maps of finite uniton number and their extended solutions into a compact Lie group to solving a sequence of ordinary differential equations in the Lie \emph{algebra}, amounting to successive integrations.  They also solve the corresponding equations in the Lie \emph{group} $\U(n)$ in some special cases of low dimension.
  
Now any compact Lie group can be embedded in $\U(n)$, but this imposes conditions on the data so that it can be hard to find, cf.\ \cite[\S 6]{unitons}.  Using the framework of \cite{burstall-guest}, we solve this problem for $\O(n)$ and give an algorithm which is inductive on dimension, finding formulae for extended solutions for the group $\O(n)$ from those for $\O(n-2)$  to end up with algebraic formulae for all harmonic maps of finite uniton number and their extended solutions from a surface to $\O(n)$ of finite uniton number in terms of free holomorphic data.  Our method is to interpret the extended solution equations in the Lie group and replace the initial data of Burstall and Guest, which had to be integrated in \cite{burstall-guest}, by data which gives the solution by differentiation and algebraic operations.

Note that it does not seem easy to extend our method to general compact Lie groups; however, a modification of our method has been developed for the symplectic group \cite{oliver} where harmonic maps and extended solutions were found in \cite{pacheco-sympl}, but with constrained holomorphic data.

The authors thank Fran Burstall, Joe Oliver, Rui Pacheco, Martin Svensson and the referee for some useful comments on this paper.

\section{preliminaries} \label{sec:prelims}

\subsection{Harmonic maps into a Lie group} \label{subsec:Lie-groups}
We recall the basic theory of harmonic maps from Riemann surfaces to Lie groups and symmetric spaces. 
 Throughout this paper, all manifolds, bundles, and structures on them, will be taken to be
$C^{\infty}$-smooth, and all manifolds will be without boundary.
\emph{Throughout this paper $M$ will denote a Riemann surface}, i.e., a 
connected $1$-dimensional complex manifold, equivalently a (smooth) oriented $2$-dimensional manifold with a conformal structure.
Since harmonicity of a map from a $2$-dimensional manifold only depends on the conformal structure
\cite[\S 4B]{eells-sampson} (see also, for example,
\cite[\S 1.2]{wood60}), the concept of harmonicity for a map from a Riemann surface is well defined.
 
 In the case of maps from a Riemann surface $M$
to a Lie group $G$, we can formulate the harmonicity equations in the following way \cite{uhlenbeck,guest-book}.
For any smooth map $\varphi:M\to G$, set $A^{\varphi}=\half \varphi^{-1}\d\varphi$;
thus $A^{\varphi}$ is a $1$-form with values in the Lie algebra $\g$ of $G$;
in fact, it is half the pull-back of the Maurer--Cartan form of $G$.
Now, any compact Lie group can be embedded in the unitary group
$\U(n)$; such an embedding is totally geodesic.  {}From the composition law
\cite[\S 5A]{eells-sampson}, a smooth map into a totally geodesic submanifold $N$ of a Riemannian manifold $P$ is harmonic into $N$ if and only if it is harmonic as a map into $P$; thus it is natural to first consider harmonic maps into $\U(n)$.  
Let $\CC^n$ denote the trivial complex bundle $\CC^n = M \times \C^n$, then
$D^{\varphi} = \d+A^{\varphi}$
defines a unitary connection on $\CC^n$.
We decompose $A^{\varphi}$ and $D^{\varphi}$ into $(1,0)$- and $(0,1)$- parts; explicitly, in a (local complex) coordinate domain
$(U,z)$,  writing
$\d\varphi = \varphi_z \d z + \varphi_{\zbar}\d\zbar$,
$A^{\varphi} = A^{\varphi}_z \d z +  A^{\varphi}_{\zbar} \d\zbar$,
$D^{\varphi} = D^{\varphi}_z \d z + D^{\varphi}_{\zbar} \d\zbar$,
$\pa_z = \pa/\pa z$ and $\pa_{\zbar} = \pa/\pa\zbar$, we have
\be{type-decomp}
A^{\varphi}_z=\half \varphi^{-1}\varphi_z\,,\quad A^{\varphi}_{\zbar}=\half \varphi^{-1}\varphi_{\zbar}\,, \quad
D^{\varphi}_z = \pa_z + A^{\varphi}_z \,,\quad
D^{\varphi}_{\zbar} = \pa_{\zbar} + A^{\varphi}_{\zbar} \,.
\ee

By the \emph{(Koszul--Malgrange) holomorphic structure \cite{koszul-malgrange} induced by $\varphi$} we mean the unique holomorphic structure on $\CC^n$ with $\dbar$-operator given on each coordinate domain $(U,z)$ by $D^{\varphi}_{\zbar}$\,;
we denote the resulting holomorphic vector bundle by $(\CC^n, D^{\varphi}_{\zbar})$.   
Uhlenbeck \cite{uhlenbeck} provided the following nice formulation of harmonicity: \emph{a smooth map $\varphi:M\to G$ is harmonic if and only if, on each coordinate domain, $A^{\varphi}_z$ is a holomorphic endomorphism of the holomorphic vector bundle $(\CC^n, D^{\varphi}_{\zbar})$}. 
We call harmonic maps $\varphi$ and $\wt{\varphi}$ with
$\wt{\varphi} = g\varphi$ for some $g \in \U(n)$ \emph{(left-)equivalent}; if $\varphi$ is replaced by an equivalent harmonic map $\wt{\varphi}$,  then all the quantities in \eqref{type-decomp} are unchanged. 

Let $\N = \{0,1,2,\ldots\}$. For any $N \in \N$ and $k \in \{0,1,\ldots,N\}$, let $G_k(\C^N)$ denote the Grassmannian of
$k$-dimensional subspaces of $\C^N$; it is convenient to write
 $G_*(\C^N)$ for the disjoint union
$\cup_{k=0,1,\ldots,N}G_k(\C^N)$.
We shall often identify, without comment, a smooth map
$\varphi:M \to G_k(\C^N)$ with the rank $k$ subbundle 
of $\CC^N = M \times \C^N$ whose fibre at $p \in M$ is $\varphi(p)$; we denote this subbundle also by $\varphi$, not underlining
this as in, for example, \cite{burstall-wood, ferreira-simoes, ferreira-simoes-wood}.

For a subspace $V$ of $\C^n$ we denote by $\pi_V$
(resp.\ $\pi_V^{\perp}$) orthogonal projection from $\C^n$ to $V$
(resp.\ to its orthogonal complement $V^{\perp}$); we use the same notation for orthogonal projection from $\CC^n$ to a subbundle.
The \emph{Cartan embedding} \cite[p.~66]{cheeger-ebin} of the complex Grassmannian is given by 
\begin{equation} \label{cartan}
\iota:G_*(\C^n)\hookrightarrow \U(n),\quad \iota(V)=\pi_V-\pi_{V}^{\perp}\,;
\end{equation}
this is totally geodesic, and isometric up to a constant factor.
We shall identify $V$ with its image $\iota(V)$; since $\iota(V^{\perp}) = -\iota(V)$,
this identifies $V^{\perp}$ with $-V$.

\subsection{Extended solutions and the Grassmannian model} \label{subsec:extd-solns}

Let $G$ be a compact connected Lie group with complexification $G^{\cc}$; denote the corresponding Lie algebras by
$\g$ and $\g^{\cc} = \g \otimes \C$.

For any Lie group, we define the \emph{free} and \emph{based loop groups} by
$\La G = \{\gamma:S^1 \to G : \ga \text{ smooth}\}$ and
$\Omega G = \{\ga \in \La G : \ga(1) = e\}$, respectively, where $e$ denotes the identity of $G$; their corresponding Lie algebras $\La\g$ and $\Omega\g$ are similarly defined.
By an \emph{extended solution} \cite{uhlenbeck} we mean a smooth map
$\Phi:M \to \Omega G$ from a (Riemann) surface which satisfies
$\Phi^{-1}\Phi_z = (1-\la^{-1})A$ on each coordinate domain $(U,z)$ for some map $A:U \to \g^{\cc}$.
We frequently write $\Phi_{\la}(z) = \Phi(z)(\la)$ \ ($z \in M$, $\la \in S^1$). 
Given an extended solution $\Phi:M \to \Omega G$, for any $g \in G$, $\varphi = g \Phi_{-1}$ is harmonic with the $A^{\varphi}_z$ of \eqref{type-decomp} equal to the $A$ just defined; $\varphi$ and $\Phi$ are said to be \emph{associated} to each other.
Any harmonic map on a \emph{simply connected} domain has an associated extended solution.
Any two extended solutions $\Phi$ and $\wt\Phi$ associated to the same or equivalent harmonic map are related by a loop: $\wt\Phi = \eta\Phi$ where $\eta \in \Omega G$: we shall say that such extended solutions are \emph{equivalent}; we are interested in finding harmonic maps and extended solutions up to equivalence. 

We specialize to $G=\U(n)$ with complexification
$G^{\cc} = \GL(n,\C)$ and corresponding Lie algebras
$\g=\u(n)$ and $\g^{\cc} = \gl(n,\C)$. 
Define the \emph{algebraic loop group} to be the subgroup $\Omega_{\alg}\U(n)$ of those $\gamma \in \Omega\U(n)$ given by finite Laurent (i.e., Fourier) series:
$\gamma = \sum_{i=s}^t \lambda^k S_k$ where $s \leq t$ are integers and the $S_k$ are $n \times n$ complex matrices, and define
$\La_{\alg}\U(n)$ similarly.  We say that $\Phi$ has \emph{finite uniton number}
if it is a map from $M$ to
$\Omega_{\alg}\U(n)$; more precisely, the uniton number is defined to be $t-s$ assuming $S_s$ and $S_t$ are non-zero. For $r \in \N$, let
$\Omega_r\U(n)$ denote the set of polynomials of degree at most $r$:
\be{Omega_r}
\Omega_r\U(n) = \big\{\ga \in \Omega_{\alg}\U(n): \ga = \sum_{k=0}^r \lambda^k S_k, \quad S_k \in \gl(n,\C) \big\}.
\ee
 
Following \cite{uhlenbeck} a harmonic map $\varphi:M \to \U(n)$ is said to be of \emph{finite uniton number} if it has an associated polynomial extended solution $\Phi:M \to \Omega_r\U(n)$.  
Then the \emph{($\U(n)$) (minimal) uniton number of\/ $\varphi$} is the minimum degree of
such a $\Phi$.  Any harmonic map from a \emph{compact} surface $M$ to $\U(n)$ which has an associated extended solution defined on the whole of $M$ is of finite uniton number at most $n-1$
\cite{uhlenbeck}; in particular, this applies to any harmonic map from $S^2$.

Now let $\H = \H^{(n)}$ denote the Hilbert space
$L^2(S^1,\C^n)$. By expanding into Fourier series, we have
$$\H = \text{ linear closure of }
	\spa\{\lambda^i e_j\ :\ i\in\Z,\ j=1,\ldots,n\},
$$
where $\{e_1=(1,0,0,\ldots,0),\, 
e_2=(0,1,0,\ldots,0), \ldots, e_n = (0,0,\ldots,0,1)\}$ is the standard basis
 for $\C^n$.
Thus, elements of $\H$ are of the form $v = \sum_i \lambda^i v_i$ where each
$v_i \in \C^n$.
If $w = \sum_i \lambda^i w_i$ is another element of $\H$, its
$L^2$ inner product with $v$ is given by $\ip{v}{w} = \sum_i v_i \ov w_i$.
The natural action of $\U(n)$ on $\C^n$ induces an action on $\H$ which is isometric with respect to this $L^2$ inner product.  We consider the closed subspace 
$$
\H_+= \H_+^{(n)} = \text{ linear closure of } \spa\{\lambda^i e_j\ :\ i \in \N,\ j=1,\dots,n\}.
$$
The action of $\Omega\U(n)$ on $\H$ induces an action on subspaces of $\H$; denote by $\Gr = \Gr^{(n)}$ the orbit of $\H_+$ under that action, see \cite{pressley-segal} for a description of that orbit.
The action gives a bijective map
\begin{equation} \label{Grass-model}
\Omega\U(n)\ni\Phi \mapsto W := \Phi\H_+\in\Gr.
\end{equation}
We will sometimes write $W_{\la} = \Phi_{\la}\H_+$ when we need to consider dependence on $\la \in S^1$.
Note that $W = \Phi\H_+$ is `shift-invariant', i.e., closed under multiplication by $\la$, indeed $\la W = \Phi\la\H_+ \subset \Phi\H_+ = W$, so that $\Phi$ gives an isomorphism between $\H_+/\la\H_+ \cong \C^n$ and
$W/\la W$.  

The map \eqref{Grass-model} restricts to a bijection from the
algebraic loop group $\Omega_{\alg}\U(n)$ to the
set of $\lambda$-closed subspaces $W$ of $\H$
satisfying $\lambda^r\H_+ \subset W\subset\lambda^{s}\H_+\,$
for some integers $r \geq s$; it further restricts to a bijection
from $\Omega_r\U(n)$ to the subset $\Gr_r \subset \Gr$ of those $\lambda$-closed subspaces $W$ of $\H$ satisfying
\be{W-finite}
\lambda^r\H_+ \subset W \subset \H_+\,.
\ee

Now let $\Phi:M \to \Omega\U(n)$ be a smooth map and set $W =\Phi\H_+:M \to \Gr$. 
We can regard $W$ as a subbundle of the trivial bundle $\ul{\H}:= M \times \H$.   Then
G.\ Segal \cite{segal} showed that $\Phi$ is an extended solution if and only if $W$ satisfies two conditions:
\be{W-extd}
\left\{ \quad
 \begin{matrix} \text{(i)} & W \text{ is holomorphic subbundle of
$\ul{\H}$\,, i.e., }
\pa_{\zbar} (\Gamma(W)) \subset \Gamma(W), \\
\text{(ii)}  & \Gamma(W)\text{ is closed under the operator } \lambda \pa_z\,, \text{ i.e., }
\lambda \pa_z \Gamma(W) \subset \Gamma(W). \end{matrix}
\right. 
\ee
Here $\Ga(\cdot)$ denotes the space of smooth sections.
We call $W=\Phi\H_+$ the \emph{Grassmannian model} of the extended solution $\Phi$.
The assignment $\Phi \mapsto W=\Phi\H_+$ induces a one-to-one correspondence between polynomial extended solutions
$\Phi:M \to \Omega_r\U(n)$ and smooth maps $W:M \to Gr_r$ satisfying \eqref{W-extd}.

\subsection{Complex extended solutions}
\label{subsec:cx-extd-solns}
Let $\La^+\U(n)^{\cc}$ (resp. $\La^*\U(n)^{\cc}$) denote the subgroup of $\La\U(n)^{\cc}$ consisting of smooth maps
$S^1 \to \U(n)^{\cc} = \GL(n,\C)$ which extend holomorphically to $\{\la \in \C: |\la| < 1\}$ (resp.\ $\{\la \in \C: 0 < |\la| < 1\}$);
$\La^+\u(n)^{\cc} = \La^+\gl(n,\C)$ is similarly defined.Following \cite{burstall-guest}, by a \emph{complex extended solution} we mean a smooth map $\Psi:M \to \La^*\U(n)^{\cc}$ which satisfies, on each coordinate domain $(U,z)$,
\be{cx-extd}
\la\Psi^{-1}\Psi_z \in \La^+\u(n)^{\cc},
\ee
and is holomorphic with respect to the complex structure induced from $\U(n)^{\cc} = \GL(n,\C)$, i.e., for fixed $\la$,
 the entries of $M \ni z \mapsto \Psi(z)(\la) \in \U(n)^{\cc}$ are holomorphic.
Recall \cite[Theorem 8.11]{pressley-segal} that the product map
$\Omega \U(n) \times \La^+\U(n)^{\cc} \to \La\U(n)^{\cc}$ is a diffeomorphism.  This gives the \emph{Iwasawa decomposition} or \emph{loop group factorization} of $\La\U(n)^{\cc}$ as the product of the two given factors.  It also gives an identification between $\Omega\U(n)$
and the homogeneous space $\La\U(n)^{\cc}/\La^+\U(n)^{\cc}$; thus $\Omega\U(n)$ acquires the structure of a complex manifold. {}From \cite{dorfmeister-pedit-wu}, given a complex extended solution $\Psi$, its projection $\Phi = [\Psi]$ onto $\Omega \U(n)$ is an extended solution; note that this is holomorphic with respect to the complex structure just defined. 
Further, the corresponding Grassmannian model $W=\Phi\H_+$ is also given by 
$W=\Psi\H_+$.
Conversely, as in \cite{dorfmeister-pedit-wu, burstall-guest}, any extended solution $\Phi$ is locally the projection of a complex extended solution.

More generally, we shall say that a meromorphic map $\Psi:M \to \La^*\U(n)^{\cc}$ is a
\emph{meromorphic complex extended solution} if it is a complex extended solution away from its poles.  Then we can extend $W=\Psi\H_+$, and so $\Phi=[\Psi]$, smoothly over the poles: indeed the columns of $\Psi$ give meromorphic sections of $W$ which span $W \!\!\mod \la W$, i.e., writing
$Y$ for the span of the columns of $\Psi$ so that $Y=\Psi(\C^n)$, then $W = \sum_{i=0}^{\infty} \la^i Y$.  
Note that $Y$, and so $W$, extend as in \cite[Lemma 4.1(ii)]{unitons}; in fact,
the columns of $\Psi$ form a \emph{meromorphic basis} for $Y$,
cf. \cite[\S 7]{dai-terng}. 
We will continue to write
$\Phi =[\Psi]$ for the projection of $\Psi$ onto $\Omega\U(n)$ even when $\Psi$ is meromorphic.

The process of finding $\Phi$ explicitly from $\Psi$ can be tricky in the general case;
however, in the finite uniton number case, $\Phi$ can be found explicitly from $W$ by the formulae in \cite{unitons}, see the next section.
Conversely, \emph{given an extended solution $\Phi:M \to \Omega\U(n)$ of finite uniton number (i.e., with values in $\Omega_{\alg}\U(n)$), 
there is a meromorphic complex extended solution $\Psi:M \to \La^*\U(n)^{\cc}$ with
$\Phi = [\Psi]$}; this follows from Proposition \ref{prop:Psi-form} below.

\subsection{Uniton factorizations from extended solutions} \label{subsec:uniton-fact}

Let $\varphi : M \to \U(n)$ be a harmonic map.  K.\ Uhlenbeck called a subbundle $\al$ of $\CC^n$ a \emph{uniton} (for $\varphi$) if
(i) $\al$ is holomorphic with respect to the Koszul--Malgrange holomorphic structure induced
by $\varphi$, i.e., 
$D^{\varphi}_{\zbar}(\sigma) \in \Ga(\al)$ for all
$\sigma \in \Ga(\al)$; and 
(ii) $\al$ is closed under the endomorphism $A^{\varphi}_z$, i.e.,
$A^{\varphi}_z(\sigma) \in \Ga(\al)$ for all
$\sigma \in \Ga(\al)$.
She showed \cite{uhlenbeck} that given a harmonic map $\varphi$ and a uniton $\alpha$, the product $\wt{\varphi} = \varphi(\pi_{\al} - \pi_{\al}^{\perp})$ gives a new harmonic map, a process she called \emph{adding a uniton}.
If $\Phi$ is an extended solution, we say that $\alpha$ is a uniton for $\Phi$ if it is a uniton for any associated harmonic map $\varphi = g\Phi_{-1}$ \ $(g \in \U(n))$; then we have
\cite[Corollary 12.2]{uhlenbeck}: 
\emph{given an extended solution $\Phi:M \to \Omega\U(n)$,
a subbundle $\al$ of\/ $\CC^n$ is a uniton for $\Phi$ 
 if and only if $\wt{\Phi} = \Phi(\pi_{\al} + \la\pi_{\al}^{\perp})$ is an extended solution}.

 Let $\Phi:M \to \Omega_r\U(n)$ be a \emph{polynomial} extended solution (see \S \ref{subsec:extd-solns}).  By a \emph{uniton factorization} of $\Phi$
we mean a product:
\be{fact}
\Phi = (\pi_{\al_1} + \la\pi_{\al_1}^{\perp}) \cdots
(\pi_{\al_r} + \la\pi_{\al_r}^{\perp})
\ee
where each $\al_i$ is a uniton for the \emph{partial product}
$\Phi_{i-1} = (\pi_{\al_1} + \la\pi_{\al_1}^{\perp}) \cdots (\pi_{\al_{i-1}} + \la\pi_{\al_{i-1}}^{\perp})$; here we set $\Phi_0 = I$.
Uhlenbeck \cite{uhlenbeck} proved that any polynomial extended solution has a uniton factorization.
A tool for finding uniton factorizations was proposed by Segal \cite{segal}, namely
that they are equivalent to certain filtrations; this was developed in \cite{unitons} where the following terminology was introduced:  Let $\HH_+$ denote the trivial bundle $M \times \H_+$\,.
By a \emph{$\la$-filtration $(W_i)$ of\/ $W$} we mean a nested sequence
$$
W = W_r \subset W_{r-1} \subset \cdots \subset W_0 = \HH_+
$$
of $\la$-closed subspaces of $\HH_+$ with 
$\la W_{i-1} \subset W_i \subset W_{i-1}$ \
$(i = 1, \ldots, r)$.
Two examples of $\la$-filtrations are the
\emph{Segal filtration} $(W_i^S$) \cite{segal} and the
\emph{Uhlenbeck filtration} $(W_i^U)$
\cite[\S 2.2]{uhlenbeck} given by $W_i^S = W + \la^i\HH_+$ and $W_i^U = (\la^{i-r}W) \cap \HH+$.
These are obtained by applying the following steps (called \emph{$\la$-steps} in \cite{unitons}) for $i = r,r-1,\ldots,2,1$, starting with $W_r^S = W_r^U = W$:
\be{segal-uhl}
W_{i-1}^S = W_i^S + \la^{i-1}\HH_+ \quad \text{and} \quad
W_{i-1}^U = (\la^{-1}W_i^U) \cap \HH_+ 
= (\la^{-1}W_i^U) \cap \HH_+ + \la^{i-1}\HH_+\,.
\ee
If we apply these steps alternately, we get a filtration called an \emph{alternating filtration}
\cite[Example 4.5]{unitons}. Starting with an Uhlenbeck step
on $W =W_r$\,, this is given by
\be{alt-fact}
W_{r-2k+1} = \la^{-k}W \cap \HH_+ + \la^{r-2k+1}\HH_+\,,
\quad
W_{r-2k} = \la^{-k}W \cap \HH_+ + \la^{r-2k}\HH_+ \quad (k=1,2,\ldots).
\ee

Let $W = \Phi\H_+$ for an extended solution $\Phi$ and let $(W_i)$ be a $\la$-filtration of $W$.  
Then \cite[\S 3]{unitons} the $W_i$ satisfy \eqref{W-extd} so $W_i = \Phi_i\H_+$ for some extended solution $\Phi_i$.  
Let $P_0:\H_+ \to \C^n$ denote evaluation at $\la=0$, i.e., $P_0(\sum \la^i L_i) = L_0$. Then \cite[Proposition 2.3]{unitons}, \emph{setting
\be{alpha}
\al_i = P_0\Phi_{i-1}^{-1}W_i \qquad (i=1,2,\ldots,r)
\ee
gives a uniton factorization \eqref{fact} with partial products given by the $\Phi_i$}\,; all uniton factorizations are given this way
\cite[\S 3]{unitons}. 
The formula \eqref{alpha} gives \emph{explicit} formulae for any uniton factorization; these include the  formulae of \cite{dai-terng,ferreira-simoes-wood} for the Segal and Uhlenbeck factorizations.  Applying \eqref{alpha} to the alternating filtration gives the \emph{alternating factorization} which has the useful property in the $\O(n)$ case
that adjacent unitons combine to give real quadratic factors, see \cite[\S 6.1]{unitons}. We shall use this factorization in \S \ref{subsec:r=1,2}ff. 
  
\subsection{Maps into complex Grassmannians and $S^1$-invariant maps}\label{subsec:cx-Grass}

Recall the Cartan embedding \eqref{cartan}.  Let $\Phi$ be an extended solution and set $W = \Phi\H_+$.
Then $\Phi$ satisfies the symmetry condition: 
\be{Phi-Grass}
\Phi_{\la}\Phi_{-1} = \Phi_{-\la} \qquad (\la \in S^1) 
\ee
if and only if $W_{-\la} = W_{\la}$ \ $(\la \in S^1)$\,.
 In this case, the corresponding harmonic map $\varphi = \Phi_{-1}$
satisfies $\varphi^2 = I$ and so is a (harmonic) map into a complex Grassmannian $G_*(\C^n)$; conversely, it follows from \cite[\S 15]{uhlenbeck} that 
\emph{any harmonic map $\varphi:M \to G_*(\C^n)$ of finite uniton number
 is of the form $\varphi = \Phi_{-1}$ for some polynomial extended solution $\Phi$} satisfying
\eqref{Phi-Grass}, see
\cite[\S 5.1]{unitons} where bounds on the degree of $\Phi$ are given.
See \cite{ferreira-simoes} for more information and explicit formulae.

As a special case of the above, an extended solution $\Phi:M \to \Omega\U(n)$ is called \emph{$S^1$-invariant} if
\be{Phi-S1-invt}
\Phi_{\la\mu} = \Phi_{\la}\Phi_{\mu} \quad (\la,\mu \in S^1),
\ee
equivalently, $W=\Phi\H_+$ satisfies $W_{\la\mu} = W_{\la}$ \ $(\la,\mu \in S^1)$\,.
Note that this implies \eqref{Phi-Grass}, so that $\Phi_{-1}$ is a harmonic map into a complex Grassmannian.
In fact (cf.\ \cite[Proposition 2.10]{unitons}),
\emph{an extended solution $\Phi$ is $S^1$-invariant if and only if it
 has a uniton factorization \eqref{fact} with nested unitons}:
\be{nested}
0=\al_0 \subset \al_1 \subset \al_2 \subset \cdots \subset \al_r \subset \al_{r+1} = \CC^n
\ee
 \emph{for some $r$}.  Further, the $\al_i$ are holomorphic subbundles of $\CC^n$ which form a
 \emph{superhorizontal sequence} (see, for example \cite[Definition 3.13]{unitons}), i.e.,
for all $i \in \{0,1,\ldots, r\}$, $\pa_z(s) \in \Ga(\al_{i+1})$ for all $s \in \Ga(\al_i)$.
The corresponding Grassmannian model $W =\Phi\H_+$ is given by
\be{Grass-S1}
W = \al_1 + \la \al_2 + \ldots + \la^{r-1} \al_r + \la^r \HH_+\,,
\ee
 and the corresponding harmonic map $\varphi = \Phi_{-1}$ is the map into a complex Grassmannian given by
\be{phi-S1-invt}
 \varphi = \sum_{i=0}^{[r/2]}\psi_{2i} \quad \text{where}
\quad \psi_i = \al_i^{\perp} \cap \al_{i+1} \quad (i=0,1,\ldots, r) \,.
\ee
The map $(\psi_i) \mapsto \varphi$ can be interpreted as a \emph{twistor fibration}, see \cite[\S 3]{burstall-guest} and \cite{burstall-rawnsley} for the general theory, \cite{twistor} for further constructions, and \S \ref{subsec:On-1} for the real case.
 
\emph{An example of an $S^1$-invariant extended solution
with $r=n-1$ is given by
setting $\al_i =$ the $(i-1)$th associated curve $f_{(i-1)}$ \cite[Definition 4.2]{unitons} of a full holomorphic map $f:M \to \CP^{n-1}$.}

\subsection{The method of Burstall and Guest for $\U(n)$}
\label{subsec:BuGu}

The starting point for the theory in \cite{burstall-guest} is a finer classification than that provided by uniton number by using `canonical elements':
Let $G$ be a compact connected Lie group with complexification $G^{\cc}$; denote the corresponding Lie algebras by
$\g$ and $\g^{\cc} = \g \otimes \C$.
Let $\de_1, \ldots, \de_{\ell}$ be a choice of simple roots
for some Cartan subalgebra $\t$.
Then a \emph{canonical element} (for $\g$) \cite{burstall-guest,burstall-rawnsley} is an element $\xi \in \t$ such that $\de_j(\xi) = 0$ or $\ii\ (\,= \sqrt{-1})$ for all $j$.
The eigenvalues of $\ad \xi$ are of the form $\ii k$ where $k$ is an integer with $-r \leq k \leq r$ where $r=r(\xi) = \max\{k : \g_k(\xi) \neq 0\}$;
we define $\g_k=\g_k(\xi)$ to be the corresponding eigenspace; we then have
$\g^{\cc} = \sum_{k=-r}^r \g_k$\,.

We now apply this to $\u(n)$:
we shall denote the eigenspace $\g_k(\xi)$ of $\ad\xi$ in $\u(n)^{\cc}=\gl(n,\C)$ by $\g_k^{\cc} = \g_k^{\cc}(\xi)$ to distinguish it from the $\o(n)$ case in \S \ref{subsec:On}.
According to \cite[Proposition A1]{burstall-eschenburg-ferreira-tribuzy}, the canonical elements of $\u(n)$ are of the form 
$\xi = \ii\diag(\xi_1+\la_0,\ldots,\xi_n+\la_0)$ where $\la_0 \in \R$ and
the $\xi_i$ are non-negative integers satisfying
\be{xi}
\xi_i - \xi_{i+1} = 0 \text{ or }1, \quad \xi_n = 0.
\ee
Note that this implies that $\xi_1 = r(\xi)$.
As in \cite[p.~562]{burstall-guest}, essentially by considering the centreless group
$\U(n)/Z(\U(n))$, we may take $\la_0 = 0$, so that
\emph{by a canonical element of\/ $\Omega_r\U(n)$ we mean a diagonal matrix
$\xi = \ii\diag(\xi_1,\ldots,\xi_n)$ where the\/ $\xi_i$ are non-negative integers satisfying \eqref{xi}}.
We have a corresponding \emph{canonical geodesic}
$\ga_{\xi}:S^1 \to \U(n)$ defined by
$\ga_{\xi}(\la) = \diag(\la^{\xi_1},\ldots,\la^{\xi_n})$, thus
$\ga_{\xi} \in \Omega_r\U(n)$.

The canonical element $\xi$ is determined by the $(r+1)$-tuple
$(t_0,t_1,\ldots,t_r)$ of positive integers where $t_j := \#\{i:\xi_i = j\}$; 
we call $(t_0,t_1,\ldots,t_r)$ the \emph{type} of $\xi$.
Note that $\sum_{j=0}^r t_j =n$; we shall see that the type determines the block structure of the $n \times n$-matrices below.  In particular,
$\g_k^{\cc} = \{B = (b_{ij}) \in \gl(n,\C): b_{ij}=0 \ \text{ if} \ \xi_i-\xi_j \neq k\}$, i.e., $\g_k^{\cc}$ consists of matrices with entries zero unless they are on the \emph{$k$th block superdiagonal: $\xi_i-\xi_j = k$} (if $k$ is negative this is below the diagonal).
As in \cite[Corollary 14.4]{uhlenbeck}, $r \leq n-1$; equality is attained by type $(1,1,\ldots,1)$, in which case
$\xi_i = r+1-i$ and $\xi_i-\xi_j = j-i$. The example at the end of \S \ref{subsec:cx-Grass} is of this type.

Write $\La^+_{\alg}\U(n)^{\cc} =  \La^+_{\alg}\GL(n,\C) := \La^+\U(n)^{\cc} \cap \La_{\alg}\U(n)^{\cc}$ and similarly for
$\La^+_{\alg}\u(n)^{\cc} =  \La^+_{\alg}\gl(n,\C)$.
To apply the above to find polynomial extended solutions, and so harmonic maps of finite uniton number into $\U(n)$, we need
\begin{definition} \label{def:A-complex}
Define a finite-dimensional Lie subgroup $\A_{\xi}^{\cc}$ of $\La^+_{\alg}\GL(n,\C)$ by 
\begin{multline*}
\A_{\xi}^{\cc} = 
\{A = (a_{ij}) \in \La^+_{\alg}\GL(n,\C):\\ a_{ij} = \delta_{ij}
\text{ if\/ } \xi_i \leq \xi_j\,,
\text{ otherwise }
a_{ij} \text{ is polynomial in } \la 
\text{ of degree at most } \xi_i-\xi_j-1\}.
\end{multline*}
\end{definition}

In the sequel, $[ \hphantom{\Psi} ]$ denotes the projection $\La\U(n)^{\cc} \to \Omega\U(n)$ onto the first factor in the Iwasawa decomposition of \S \ref{subsec:cx-extd-solns}.

\begin{proposition} \label{prop:Psi-form}
Let\/ $\wt\Phi:M \to \Omega_{\wt r}\U(n)$ be a polynomial extended solution for some $\wt r \in \N$.  Then there is an equivalent extended solution $\Phi:M \to \Omega_r\U(n)$ with
$0 \leq r \leq \wt r$, a canonical element
$\xi =\ii\diag(\xi_1,\ldots,\xi_n)$  of\/ $\Omega_r\U(n)$
and a meromorphic map $A:M \to \A_{\xi}^{\cc}$ such that $\Phi = [A\ga_{\xi}]$.

Further, $A$ and $\xi$ are uniquely determined by $\Phi$.

All harmonic maps $\varphi:M \to \U(n)$ of finite uniton number have such an associated extended solution
$\Phi$\,.  
\end{proposition}

Given a canonical element $\xi$ of type $(t_0,\ldots, t_r)$, we shall say that $A:M \to \A_{\xi}^{\cc}$,  $\Phi = [A \ga_{\xi}]$ and the associated Grassmannian model $W=\Phi\H_+$
are of \emph{canonical type}, specifically, \emph{of type $\xi$}, or \emph{of type $(t_0,\ldots, t_r)$}.
Note that $\Psi=A\ga_{\xi}$ is a meromorphic extended solution with $\Phi=[\Psi]$, see \S \ref{subsec:cx-extd-solns},
and $\Phi$ and $\Psi$ are both polynomial of degree $r$ in $\la$.
 
\begin{proof}
Define a finite-dimensional Lie subalgebra $\aaa_{\xi}^{\cc}$ of
 $\La^+_{\alg}\gl(n,\C)$ by
\begin{multline} \label{u0xi}
\aaa_{\xi}^{\cc} =
\{b = (b_{ij}) \in \La^+_{\alg}\gl(n,\C): \\
b_{ij} = 0 \text{ if\/ } \xi_i \leq \xi_j\,,
\text{ otherwise }
b_{ij} \text{ is polynomial in } \la 
\text{ of degree at most } \xi_i-\xi_j-1\};
\end{multline}
this is the $\u^0_{\xi}$ of \cite[Proposition 2.7]{burstall-guest} for $\g =\u(n)$.
It is the Lie algebra of $\A_{\xi}^{\cc}$ and
the exponential map
$B \mapsto A = \exp B = \sum_{i=0}^{\infty} B^i/i!$
maps $\aaa_{\xi}^{\cc}$ to $\A_{\xi}^{\cc}$. 
{}From \cite[Theorem 4.5 and p.~560]{burstall-guest}, 
given $\wt\Phi$, there is an equivalent extended solution
$\Phi:M \to \Omega_r\U(n)$, canonical element $\xi =\ii\diag(\xi_1,\ldots,\xi_n)$ of $\Omega_r\U(n)$ and discrete subset\/ $D$ of $M$ such that a complex extended solution
$\Psi:M \setminus D \to \La^+_{\alg}\GL(n,\C)$ with
$[\Psi]=\Phi$ is given by $\Psi = A \ga_{\xi}$ where
$A = \exp B$ for some holomorphic map $B:M \setminus D \to \aaa_{\xi}^{\cc}$\,; thus $A$ is a holomorphic map from  $M \setminus D$
to $\A^{\cc}_{\xi}$\,.
Uniqueness of $\xi$ is from the Bruhat decomposition,	 cf.\ \cite[Corollary 2.2]{burstall-guest}; uniqueness of $B$ and so $A$ follows from \cite[Proposition 2.7]{burstall-guest}.
Alternatively, Suppose $[\wt{A}\ga_{\xi}] = [A\ga_{\xi}]$ for $A,\wt{A}:M \setminus D \to \A_{\xi}^{\cc}$. Then $\wt{A}\ga_{\xi} = A\ga_{\xi}B$ for some
$B: M \to \La^+\U(n)^{\cc}$.   Then $B = \ga_{\xi}^{-1}\wt{B}\ga_{\xi}$ where $\wt{B} = A^{-1}\wt{A}$;  the matrix $B$ is the product of block upper-triangular matrices, so is block upper-triangular, i.e.\  
$b_{ij} = \delta_{ij}$ ($\xi_i \leq \xi_j$).  On the other hand, the entries of $B$ below the block diagonal are given by $b_{ij} = \la^{\xi_j - \xi_i}\wt{b}_{ij}$ ($\xi_i > \xi_j$) which, since $\wt{B} \in \A_{\xi}^{\cc}$, has degree at most
$(\xi_j - \xi_i) + (\xi_i - \xi_j-1) = -1$, a contradiction to $B$ having values in $\La^+\U(n)^{\cc}$ unless $b_{ij} = 0$.
Hence $B = I$ and uniqueness is established.

Since $\Phi:M \to \Omega_r\U(n)$ is holomorphic map to a projective algebraic variety,
$B$, and so $A$ and $\Psi = A\ga_{\xi}$, are meromorphic on $M$ as in \cite[p.~560]{burstall-guest}.

All harmonic maps of finite uniton number have a polynomial associated extended solution $\wt\Phi:M \to \Omega_{\wt r}\U(n)$,
and so an associated extended solution $\Phi:M \to \Omega_r\U(n)$
given as described.
\end{proof}

\begin{remark} \label{rem:A}
(i) The method of Burstall and Guest applies to centreless groups, see
\cite{correia-pacheco} for a study of extended solutions into groups with centre, using a related notion of `$I$-canonical element'.

(ii) The matrices $B$ in $\aaa_{\xi}^{\cc}$ are nilpotent, and the matrices $A$ in $\A_{\xi}^{\cc}$ are
\emph{block unitriangular} by which we mean upper block-triangular with identity matrices on the block diagonal;  in particular $A-I$ is nilpotent.
The exponential map $B \mapsto A = \exp B$ is given by a finite power series in $B$; further, it is surjective with inverse given  $A \mapsto \log A$, a finite power series in $A-I$.
  
(iii) We exemplify the form of $A$ by showing it for types $(1,1,1,1,1,1)$ (so $r=5$) and $(1,2,2,1)$ (so $r=3$), respectively: the superscript in the notation $a_{ij}^{[k]}$ show the maximum degree $\xi_i-\xi_j-1$ of the polynomial $a_{ij}$; observe that this equals
$k-1$ on the $k$th block superdiagonal ($k=1,2,\ldots,r$): 
\be{A-ex}
A = \left(\begin{array}{c|c|c|c|c|c}
	1 & a_{12}^{[0]} & a_{13}^{[1]} & a_{14}^{[2]} & a_{15}^{[3]} & a_{16}^{[4]} \\ \hline
	0 & 1 & a_{23}^{[0]} & a_{24}^{[1]} & a_{25}^{[2]} & a_{26}^{[3]} \\ \hline  
	0 & 0 & 1 & a_{34}^{[0]} & a_{35}^{[1]} & a_{36}^{[2]} \\ \hline
	0 & 0 & 0 & 1 & a_{45}^{[0]} & a_{46}^{[1]}  \\ \hline
	0 & 0 & 0 & 0 & 1 & a_{56}^{[0]} \\ \hline
	0 & 0 & 0 & 0 & 0 & 1 	
\end{array}\right),
\qquad
A = \left(\begin{array}{c|cc|cc|c}
	1 & a_{12}^{[0]} & a_{13}^{[0]} & a_{14}^{[1]} & a_{15}^{[1]} & a_{16}^{[2]} \\ \hline
	0 & 1 & 0 & a_{24}^{[0]} & a_{25}^{[0]} & a_{26}^{[1]} \\   
	0 & 0 & 1 & a_{34}^{[0]} & a_{35}^{[0]} & a_{36}^{[1]} \\ \hline
	0 & 0 & 0 & 1 & 0 & a_{46}^{[0]}  \\
	0 & 0 & 0 & 0 & 1 & a_{56}^{[0]}  \\ \hline
	0 & 0 & 0 & 0 & 0 & 1 	
\end{array}\right).
\ee

(iv) The Grassmannian model $W = \Phi \H_+$ is given by $W = A \ga_{\xi}\H_+$
and so by \eqref{Grass-S1}
where $\al_i$ is the span of the columns $c_j$ of $A$ with $\xi_j < i$ (these $\al_i$ are functions of $\la$ as well as of points of $M$); clearly, the $\al_i$ are nested.  The columns of the matrix $A$ provide a canonical (a sort of `reduced echelon form') meromorphic basis for $Y = A \ga_{\xi}\C^n$ (and so for $W$), adapted to the nested sequence $(\al_i)$.
In the $S^1$-invariant case, the $\al_i$ do not depend on $\la$ and are the subbundles (2.14).

(v) $\Phi = [A\ga_{\xi}]$ satisfies the symmetry condition \eqref{Phi-Grass}, and so $\Phi_{-1}$ is a harmonic map into a Grassmannian, if and only if $A$ is a function of $\la^2$, i.e., its entries only involve polynomials with even powers of $\la$. 
Further $\Phi$ is $S^1$-invariant if and only $A$ is independent of $\la$.
Both statements follow from (iv), \S \ref{subsec:cx-Grass}, and the uniqueness of $A$.
\end{remark}

\smallskip

We now give a converse to Proposition \ref{prop:Psi-form}.
As above, denote the columns of $A$ by $c_1,\ldots, c_n$ so that
$c_j = (a_{1j},\ldots, a_{nj})^{\T}$.   We write  
$\sum_{j:\,P(j)}$ to mean the sum over all $j$ satisfying the condition $P(j)$; for example, $\sum_{j:\,\xi_j > \xi_k}$ means the sum over all columns $c_j$ in the blocks to the left of the block containing $c_k$.
Primes ${}'$ denote derivatives with respect to any local complex coordinate on $M$.
Recall the concept of `complex extended solution' from \S \ref{subsec:cx-extd-solns}. 

\begin{proposition} \label{prop:GrM}
Let $\xi$ be a canonical element of\/ $\Omega_r\U(n)$.
Let $A:M \to \A_{\xi}^{\cc}$ be a holomorphic map,
and set $\Psi = A\ga_{\xi}$.
Then $\Psi$ is a complex extended solution if and only if the columns of\/ $A$ satisfy
\be{GrM}
c_k' = \sum_{j:\,\xi_j > \xi_k} \la^{\xi_j-\xi_k -1} \rho_{jk}'c_j
		\qquad (r > \xi_k \geq 0)
\ee
where $\rho_{jk}$ is the coefficient of the term of degree $\xi_j-\xi_k -1$ in $a_{jk}$.

This equation is equivalent to
\be{GrM1}
a_{ik}' = \sum_{j:\,\xi_i \geq \xi_j > \xi_k}
		\la^{\xi_j-\xi_k -1} \rho_{jk}' a_{ij} \qquad (r \geq \xi_i > \xi_k \geq 0).
\ee

The equation \eqref{GrM1} holds if and only if it holds
 $\mod \la^{\xi_i-\xi_k-1}$ and is equivalent to
\be{GrM2}
a_{ik}' = \sum_{j:\,\xi_i > \xi_j > \xi_k} \la^{\xi_j-\xi_k -1} \rho_{jk}'a_{ij}
\qquad \mod \la^{\xi_i-\xi_k-1} \qquad
(r \geq \xi_i > \xi_k+1 \geq 1).
\ee 
\end{proposition}

We shall call any of the above three equations the 
\emph{extended solution equation (for $A$)}.

\begin{proof}
On a coordinate domain $(U,z)$, set
\be{P}
P = \la\Psi^{-1}\Psi_z\,, \quad \text{equivalently,} \quad
\la\Psi_z = \Psi P.
\ee
Then $P$ is algebraic, i.e., its entries $p_{jk}$ are polynomial in $\la$ and $\la^{-1}$ (with coefficients holomorphic in $z$); further, from the block structure of $A$, $P$ is strictly upper block-triangular, i.e., $p_{jk} = 0$ for $\xi_j \leq \xi_k$,
so \eqref{P} reads 
\be{ck}
\la (\la^{\xi_k}c_k)' = \sum_{j:\,\xi_j > \xi_k} p_{jk} \la^{\xi_j}c_j\,,
\text{ equivalently, } 
c_k' = \sum_{j:\,\xi_j > \xi_k} \la^{\xi_j-\xi_k -1}p_{jk}\,c_j \qquad (r > \xi_k \geq 0).
\ee 
Taking the $i$th row, since $A$ is block unitriangular,
$a_{ij} = 0$ for $\xi_i < \xi_j$, so \eqref{ck} is equivalent to
\be{aik}
a_{ik}' = \sum_{j:\,\xi_i \geq \xi_j > \xi_k}
		\la^{\xi_j-\xi_k -1} p_{jk} a_{ij}
\qquad (r \geq \xi_i > \xi_k \geq 0).
\ee

\emph{Suppose that\/ $\Psi$ is a complex extended solution.}
Then, from \eqref{cx-extd}, each $p_{ik}$ is polynomial in $\la$ (with no $\la^{-1}$).
We prove by induction on $\xi_i-\xi_k$ that
(*): \emph{each $p_{ik}$ is of degree $0$ and equals $\rho_{ik}'$}.

First, if $\xi_i-\xi_k =1$, since $a_{ij} = \delta_{ij}$ when $\xi_j=\xi_i$, \eqref{aik}
reads $a'_{ik} = p_{ik}$, which establishes (*) since $a'_{ik}$ has degree $0$.
 
Now suppose that (*) holds for $\xi_i-\xi_k \leq s$ for some $s \geq 1$.  Then
for $\xi_i-\xi_k =s+1$, \eqref{aik} reads
$$
a_{ik}' = \la^s p_{ik}
	+ \sum_{j:\,\xi_i > \xi_j > \xi_k}\la^{\xi_j-\xi_k -1} p_{jk} a_{ij}. 
$$
By the induction hypothesis, all the terms in the sum have degree at most
$(\xi_j-\xi_k-1) + 0 + (\xi_i-\xi_j-1) = \xi_i-\xi_k-2 = s-1$ whereas the left-hand side
$a_{ik}'$ has degree at most $\xi_i-\xi_k-1 = s$.  Then equating coefficients of degree $\geq s$
establishes (*) for $\xi_i-\xi_k =s+1$, and the induction step is complete. 

Equation \eqref{GrM} follows.
Equation \eqref{GrM1} is the $i$th row of \eqref{GrM} and so is equivalent to it.

Now, by definition of $\rho_{ik}$, the term of maximum possible degree $\xi_i-\xi_k-1$ on the left-hand side of \eqref{GrM1} equals the term of that degree, $\la^{\xi_i-\xi_k -1}\rho'_{ik}a_{ii} =  \la^{\xi_i-\xi_k -1}\rho'_{ik}$, on the right-hand side
--- all other terms in that sum are of degree
at most $(\xi_j-\xi_k-1)+(\xi_i-\xi_j-1) = \xi_i-\xi_k-2$.  Hence \eqref{GrM1} holds if and only if it holds $\mod \la^{\xi_i-\xi_k-1}$, and we can miss out terms of degree
$\xi_i-\xi_k-1$, i.e., those with with $\xi_i = \xi_j$, in the summation.  In particular, \eqref{GrM1} is equivalent to \eqref{GrM2}.

\smallskip  

\emph{Conversely, suppose that \eqref{GrM} holds.}
Then \eqref{P} holds with each $p_{jk}$ polynomial in $\la$, so that \eqref{cx-extd} holds and $\Psi$ is a complex extended solution. 
\end{proof}

Let $\xi$ be a canonical element of $\Omega_r\U(n)$ and let $(\A_{\xi}^{\cc})_0 = \A_{\xi}^{\cc} \cap \U(n)$, the group of block unitriangular $n \times n$ matrices with complex entries.
Let $\Sol_{\xi}^{\cc}$ (resp.\ $(\Sol_{\xi}^{\cc})_0$) denote the space of meromorphic maps $A$ from $M$ to $\A_{\xi}^{\cc}$ (resp.\ $(\A_{\xi}^{\cc})_0$) which satisfy the equation \eqref{GrM} away from the poles of $A$.  Combining Propositions \ref{prop:Psi-form} and
\ref{prop:GrM}, we have
 
\begin{corollary} \label{cor:one-one}
Let $\xi$ be a canonical element of\/ $\Omega_r\U(n)$. The assignment $A \mapsto \Phi=[A\ga_{\xi}]$ defines a one-to-one correspondence between $\Sol_{\xi}^{\cc}$ and the space of extended solutions\/ $\Phi:M \to \Omega_r\U(n)$ of type $\xi$.  It restricts to a one-to-one correspondence between $(\Sol_{\xi}^{\cc})_0$ and
the space of\/ $S^1$-invariant extended solutions\/ $\Phi:M \to \Omega_r\U(n)$ of type $\xi$.
\qed \end{corollary}

\begin{remark} \label{rem:GrM}
(i) In \eqref{GrM1}, we take the sum from the diagonal block onwards, as the entries $a_{ij}$ are zero to the left of that block.  However, since we only need this equation to hold $\mod \la^{\xi_i-\xi_k-1}$, we may additionally omit any entries in that diagonal block; in \eqref{GrM2},
we omit \emph{all} such entries.

(ii) An extended solution $\Phi$ of some type $\xi$ can be deformed to an $S^1$-invariant solution of the same type, called its \emph{$S^1$-invariant limit}, see \cite[\S 2]{burstall-guest}, and  \cite{aleman-martin-persson-svensson} for a treatment of smoothness.  For any $\mu \in \C$, define $A_{\mu}:M \to \A^{\C}_{\xi}$ by $A_{\mu}(z)(\la) = A(z)(\mu\la)$ \ $(z \in M, \ \la \in \C)$. If $A$ satisfies the extended solution equation \eqref{GrM}, so does  $A_{\mu}$ for all $\mu  \in \C$ including $\mu=0$.  Then the deformation is implemented by $\mu \mapsto A_{\mu}$ with $\mu$ going from $1$ to $0$.

(iii) As in Remark \ref{rem:A}(iv), the Grassmannian model $W = \Phi\H_+$ is given by \eqref{Grass-S1} where 
$\al_i = \spa\{c_j\,: \xi_j < i\}$ .  In the above deformation, these $\al_i$ tend to the unitons \eqref{nested} of the $S^1$-invariant limit.
\end{remark}

\smallskip   

The equations \eqref{GrM} for $\U(n)$ are easy to solve, see \cite[\S 4]{burstall-guest} and \cite[Ch.~22]{guest-book}.  However, finding all solutions in $\O(n)$ is not so easy: we turn to that problem now.

\section{Harmonic maps of finite uniton number into $\O(n)$} \label{sec:On}

\subsection{Generalities on harmonic maps into $\O(n)$ and its symmetric spaces}
\label{subsec:On-1}

Let $z=x+\ii y \mapsto \ov{z} = x-\ii y$ denote standard complex conjugation on $\C$.  To adapt the theory of the last section to $\O(n)$, we include $\R^n$ in $\C^n$ so that
$\R^n = \{(z_1,\ldots,z_n) \in \C^n: \ov{z}_i = z_i \ (i=1,2,\ldots, n)\}$, 
and then $\O(n)$ is the subgroup of $\U(n)$ given by
$\O(n) = \{A\in \U(n) : \ov{A} = A\} = \{A\in \U(n) : A^{\T} A = I\}$ where, for $A=(a_{ij})$,
we have $\ov{A} = (\ov{a_{ij}})$ and $A^{\T} = (a_{ji})$.
Similarly 
$\Omega\O(n)  = \Omega\SO(n) = \{\Phi \in \Omega\U(n): \ov{\Phi} = \Phi\} 
= \{\Phi \in \Omega\U(n) : \Phi^{\T}\,\Phi = I\}$
where, for $\Phi = \sum \la^i \Phi_i$, we set
$\ov{\Phi} = \sum \la^{-i}\ov{\Phi_i}$ and $\Phi^{\T} = \sum \la^i \Phi_i^{\T}$\,.

Now, given $\Phi \in \Omega\U(n)$, set $W = \Phi\H_+$ as in 
\eqref{Grass-model}. Then \cite[\S 8.5]{pressley-segal}, $\Phi \in \Omega\O(n)$ if and only if $\ov{W}^{\perp} = \la W$.  
 However, to deal with polynomial extended solutions, as in \cite{unitons} we define for each $r \in \N$
the following subset of $\Omega_r\U(n)$ (cf.\ \eqref{Omega_r}):
\be{Omega-R}
\Omega_r\U(n)^{\rr}
 = \{\Phi \in \Omega_r\U(n) : \ov{\Phi} = \la^{-r}\Phi\}
  = \{\Phi \in \Omega_r\U(n) : \Phi^{\T}\,\Phi = \la^r I\}.
\ee
Then (cf.\ \cite[\S 6]{unitons}),
\emph{$\Phi \in \Omega_r\U(n)^{\rr}$
 if and only if\/
$\ov{W}^{\perp} = \la^{1-r}W$}, in which case we say that
$\Phi$ and $W$ are \emph{real of degree $r$}. 

Let $\Phi:M \to \Omega_r\U(n)^{\rr}$ be an extended solution, and set $W = \Phi\H_+$\,.   
\emph{If $r$ is even}, then $\ov{\Phi_{-1}} = \Phi_{-1}$ so that
$\varphi = \pm\Phi_{-1}$ are harmonic maps into $\O(n)$. 
By \cite[Lemma 6.4]{unitons}, 
all harmonic maps $M \to \O(n)$ of finite uniton number have an extended solution 
 $\Phi:M \to \Omega_r\U(n)^{\rr}$ with $r$ even, and $\varphi = \pm\Phi_{-1}$ --- note that the (minimal) uniton number of $\varphi$ may be less than $r$ and may be even or odd. 
\emph{If $r$ is odd}, then, following \cite[\S 6.3]{unitons},
$n$ must be even, say $n = 2m$, and $\ov{\Phi_{-1}} = -\Phi_{-1}$
so that $\varphi = \pm \ii\Phi_{-1}$ are maps into $\O(2m)$.
In all cases, the alternating factorization \cite[\S 6.1]{unitons} of $\Phi$, which can be calculated from $W$ by 
\eqref{alt-fact}, \eqref{alpha} and \eqref{fact}, gives an explicit factorization into unitons.

The \emph{symmetric spaces} of $\O(n)$ and $\SO(n)$ are the real Grassmannians $G_k(\R^n) = \O(n)/\O(k) \times \O(n-k) = \SO(n)/{\mathrm S}(\O(k) \times \O(n-k))$ with double cover the Grassmannian of oriented subspaces, $\SO(n)/\SO(k) \times \SO(n-k)$ ($k=0,1,\ldots, n$), and, when $n=2m$, the space $\O(2m)/\U(m)$ of orthogonal complex structures $J$ on $\R^{2m}$ and its identity component $\SO(2m)/\U(m)$.  Note that mapping each $J$ to its $\ii$-eigenspace identifies $\O(2m)/\U(m)$ with the space of all maximally isotropic subspaces of $\C^{2m}$.
 Let $\Phi:M \to \Omega_r\U(n)^{\rr}$ be an extended solution which satisfies
the symmetry condition \eqref{Phi-Grass}.
\emph{If\/ $r$ is even}, 
$\varphi = \pm\Phi_{-1}$ are harmonic maps of finite uniton number into a real Grassmannian
$G_*(\R^n)$, all such harmonic maps can be obtained this way \cite[Lemma 6.6]{unitons};
note that $-\varphi = \varphi^{\perp}$.
If\/ $r$ is odd, then $n$ is even, and
$\pm\Phi_{-1}$ define harmonic maps of finite uniton number into $\O(2m)/\U(m)$ for $m=n/2$; all such harmonic maps are obtained this way
\cite[Lemma 6.9]{unitons}.

Lastly, let\/ $\Phi:M \to \Omega_r\U(n)^{\rr}$ be an extended solution
which is $S^1$-invariant, i.e., satisfies \eqref{Phi-S1-invt}.  Then $\Phi$ is given by \eqref{fact}
for some superhorizontal sequence \eqref{nested} of holomorphic subbundles of $\CC^n$ which is \emph{real} in the sense that
\emph{the \emph{polar} $\al_i^{\circ} := \ov{\al_i}^{\perp}$ of\/ $\al_i$ is $\al_{r+1-i}$ for all $i$, equivalently, with
$\psi_i$ defined by \eqref{phi-S1-invt}, $\psi_i = \ov{\psi_{r-i}}$ for all $i$}, see, for example, \cite[\S 6.4]{unitons}.
The corresponding harmonic map $\varphi:=\Phi_{-1}$ is given by \eqref{phi-S1-invt}; it defines a map into a \emph{real}  Grassmannian (resp.\ $\O(2m)/\U(m)$ with $n=2m$) according as $r$ is even (resp.\ odd).

\subsection{Analysis of harmonic maps into $\O(n)$}
\label{subsec:On}

To analyse further harmonic maps into $\O(n)$,
we equip $\C^n$ with its standard symmetric inner product
$(x,y) = \sum_{i=1}^n x_i y_i$ for $x = (x_1,\ldots,x_n)$,
$y = (y_1,\ldots,y_n)$.  Then
the \emph{complexification} $\O(n,\C)$ of $\O(n)$ is given by
$\{ A \in \GL(n,\C): A^{\T} A = I\}$ where $A^{\T}$ is the linear map characterized by
$(Ax,y) = (x,A^{\T} y)$ \ $(x,y \in \C^n)$. 
With respect to the standard basis $\{e_1=(1,0,0,\ldots,0),\, 
e_2=(0,1,0,\ldots,0), \ldots, e_n = (0,0,\ldots,0,1)\}$, the matrix for $A^{\T}$ is the usual transpose $(a_{ji})$ obtained from the matrix $A = (a_{ij})$ by reflection in the principal diagonal
 $i=j$.   However, calculations are aided by taking a \emph{null basis}
$\{\wt{e}_i\}$ for $\C^n$, i.e., one with $(\wt{e}_i,\wt{e}_j) = \delta_{i\jbar}$ where, for any $j \in \{1,\ldots, n\}$ we write
$\jbar = n+1-j$.  Such a basis is given by 
$\wt{e}_j = (1/\sqrt{2})(e_j + \ii e_{\jbar})$,
$\wt{e}_{\jbar} = (1/\sqrt{2})(e_j - \ii e_{\jbar})$ 
for $j \leq n/2$, together with $\wt{e}_{(n+1)/2} = e_{(n+1)/2}$ if $n$ is odd.
\emph{{}From now on, we shall write all vectors and
matrices with respect to this null basis};
then \emph{the standard symmetric bilinear inner product on $\C^n$
of\/ $v = \sum_j v_j \wt{e}_j$ and $w = \sum_j w_j \wt{e}_j$ is given by $(v,w) = \sum_{j=1}^n v_j w_{\jbar}$}.  In this null basis the transpose $A^{\T}$ is represented by the matrix $A^{\TT}$ with entries
$(A^{\TT})_{ij} = a_{\jbar\ibar}$; we shall call this the \emph{second transpose of $A$}.  
This definition makes sense for any (rectangular) matrix; for a square matrix $A$,
$A^{\TT}$ is obtained from $A$ by reflection in the \emph{second diagonal} $i = \ov{j}$.

As before, denote the $i$th column of $A$ by $c_i$.  Then $A \in \O(n,\C)$ if and only if
\be{cx-orthog}
(c_i,c_j) = \delta_{i \jbar} \qquad (i,j=1,\ldots, n).
\ee
Now, according to  \cite{burstall-personal}, the canonical elements of $\o(n)$ are of the form  $\xi = \ii\diag(\xi_1,\ldots,\xi_n)$ where $\xi_i$ are
integers or half-integers with $\xi_i - \xi_{i+1} = 0$ or $1$, $\xi_1=r/2$ for some $r=r(\xi) \in \N$ and 
$\xi_{\ibar} = -\xi_i$ $\forall i$, which satisfy the \emph{rider} (R): \emph{if $r$ is odd, $\#\{i:\xi_i =1/2\} \geq 2$}.   This corrects \cite[Proposition A.2]{burstall-eschenburg-ferreira-tribuzy} which omits the rider and gives a condition (C2) which is incorrect in the $\o(n)$ case.  
The corresponding eigenspaces of $\ad\xi$, which we shall denote by
$\g_k^{\rr} = \g_k^{\rr}(\xi)$, are the intersections with $\o(n,\C) = \o(n) \otimes \C$ of the eigenspaces
$\g_k^{\cc} = \g_k^{\cc}(\xi)$ for $\u(n)$ described in \S \ref{subsec:BuGu}, thus $\g_k^{\rr}$ is the $k$th block superdiagonal $\xi_i-\xi_j = k$ of $\o(n,\C)$.

When the $\xi_i$ are half-integers, the canonical elements above 
do not exponentiate to geodesics in $\O(n)$. 
However, we can work in $\Omega_r\U(n)^{\R}$ by adding the constant matrix $(r/2)I$ on to each canonical element (cf. \S \ref{subsec:BuGu})  to give the following definition.	

\begin{definition} \label{def:can-el-R}
By a \emph{canonical element of\/ $\Omega_r\U(n)^{\rr}$} we mean a diagonal matrix $\xi = \ii\diag(\xi_1,\ldots,\xi_n)$ where the $\xi_i$ are \emph{integers} with
$\xi_i - \xi_{i+1} = 0$ or $1$, $\xi_1 = r$, $\xi_n = 0$,
$\xi_{\ibar} = r-\xi_i$ and, if $r$ is odd, we have that
(R): $\xi_{n/2-1}=\xi_{n/2}$.
\end{definition}

Recall that, if $r$ is odd, $n$ is even.  In this case, the rider (R) says
$\xi_{n/2-1}=\xi_{n/2} = (r+1)/2$ and
$\xi_{n/2+1}=\xi_{n/2+2} = (r-1)/2$.
Noting that the canonical elements of $\Omega_r\U(n)^{\rr}$ form a subset of those in
$\Omega_r\U(n)$, we may define `type' as in \S \ref{subsec:BuGu}.  Then the possible types of canonical elements for $\Omega_r\U(n)^{\rr}$ are
$(t_0,t_1, \ldots,t_r)$ where the $t_i$ are positive integers such that $t_i = t_{r-i}$ for all $i$, and (by the rider (R))
if $r$ is odd, the two middle entries $t_{(r-1)/2} = t_{(r+1)/2}$ are at least $2$.

\begin{remark} \label{rem:type111} (i) \emph{When the type is $(1,t_1,\ldots,t_{r-1},1)$,
$\g_r^{\rr}$ is zero.}  Indeed, it consists of matrices with only possible non-zero entry in the top-right position, but this is zero by the skew-symmetry ($B^{\TT} = -B$) of matrices $B$ in $\o(n,\C)$.

(ii) If $n$ is odd, the maximal uniton number is $n-1$ attained by type $(1,1,\ldots,1)$. If $n$ is even, the rider (R) shows that this type is not possible, and the maximal uniton number is $n-2$ attained by type $(1,\ldots,1,2,1,\ldots,1)$.  This confirms the bounds
on the uniton number in \cite[Proposition 6.17]{unitons};
we shall see how to construct extended solutions of all types in Theorem \ref{th:gen}. 
\end{remark}

Let $\xi$ be a canonical element of $\Omega_r\U(n)^{\rr}$.
Recall the space $\A_{\xi}^{\cc}$ from Definition \ref{def:A-complex}, and set $\A_{\xi}^{\rr} = \A_{\xi}^{\cc} \cap \Omega\O(n,\C)$.
Let $A \in \A_{\xi}^{\rr}$. By definition of $\A_{\xi}^{\cc}$, each entry $a_{ij}$ of $A$ above the block diagonal, i.e., with
$\xi_i - \xi_j \geq 1$\,, is polynomial of degree at most $\xi_i-\xi_j-1$. 
We now show that when $A \in \A_{\xi}^{\rr}$, the degrees of the entries $a_{i\ibar}$ on the second diagonal which lie above the
block \emph{super}diagonal, i.e.\ with $\xi_i - \xi_{\ibar} \geq 2$, are at most one less than this.

\begin{lemma} \label{lem:top-right-deg}
Let $A \in \A_{\xi}^{\rr}$.
The degree of an element $a_{i\ibar}$ of\/ $A$ with
$\xi_i-\xi_{\ibar} \geq 2$ is at most $\xi_i-\xi_{\ibar}-2 = r-2\xi_{\ibar}-2$.
\end{lemma}

\begin{proof} 
Complex-orthogonality \eqref{cx-orthog} gives $(c_{\ibar},c_{\ibar})=0$.
When $\xi_i-\xi_{\ibar} \geq 2$, expanding this gives
$2 a_{i\ibar} = -\sum_{\ell=i+1}^{\ibar-1} a_{\ell\ibar} a_{\ellbar\ibar}$ which equals
$-\sum_{\ell:\xi_{\ibar} < \xi_{\ell} < \xi_i} a_{\ell\ibar} a_{\ellbar\ibar}$\,,
since $a_{\ell\ibar}=\delta_{\ell\ibar}$ when $\xi_{\ell}=\xi_{\ibar}$.
The degree of each product in this sum is at most
$(\xi_{\ell}-\xi_{\ibar}-1) + (\xi_{\ellbar}-\xi_{\ibar}-1)$, which gives the stated bound.
\end{proof}

We now give a version of Proposition \ref{prop:Psi-form} for $\O(n)$.
Let $\Sol_{\xi}^{\rr}$ (resp.\ $(\Sol_{\xi}^{\rr})_0$\,) denote the space of meromorphic maps
$A$ from $M$ to $\A_{\xi}^{\rr}$ (resp.\ $(\A_{\xi}^{\rr})_0$\,) which satisfy the extended solution equation \eqref{GrM} away from the poles of $A$.

\begin{proposition} \label{prop:BuGu-On}
Let $\wt\Phi: M \to \Omega_{\wt r}\U(n)^{\rr}$ be a polynomial extended solution for some $\wt r \in \N$. Then
there is an equivalent extended solution $\Phi:M \to \Omega_r\U(n)^{\rr}$ with
$0 \leq r \leq \wt r$, a
canonical element\/ $\xi =\ii\diag(\xi_1,\ldots,\xi_n)$ 
 of\/ $\Omega_r\U(n)^{\rr}$ and meromorphic map $A:M \to \A_{\xi}^{\rr}$ such that $\Phi = [A\ga_{\xi}]$.
 
Further, $A$ and $\xi$ are uniquely determined by $\Phi$, in fact,
 the assignment $A \mapsto \Phi=[A\ga_{\xi}]$ defines a one-to-one correspondence between $\Sol_{\xi}^{\rr}$ and the space of extended solutions $\Phi:M \to \Omega_r\U(n)^{\rr}$ of type $\xi$.  It restricts to a one-to-one correspondence between $(\Sol_{\xi}^{\rr})_0$ and the space of\/ $S^1$-invariant extended solutions $\Phi:M \to \Omega_r\U(n)^{\rr}$ of type $\xi$.

All harmonic maps of finite uniton number $\varphi:M \to \O(n)$ have an associated extended solution $\Phi \in \Sol_{\xi}^{\rr}$ for some canonical element $\xi$. 
\end{proposition}

\begin{proof}
Let $C$ be the centre of $\SO(n)$, this is trivial if $n$ is odd and $\{\pm I\}$ if $n$ is even; let $\pi:\SO(n) \to \SO(n)/C$ be the natural projection.  Recall that, if $\wt r$ is odd, then $n$ is even \cite[\S 6.3]{unitons}.
Then, in all cases, $\la^{-\wt r/2}\,\pi\circ\wt\Phi:M \to \Omega(\SO(n)/C)$ is an extended solution.
We apply \cite[Theorem 4.5]{burstall-guest} to the centreless group $\SO(n)/C$ which has Lie algebra $\o(n)$.
We set $\aaa_{\xi}^{\rr}$ equal to the intersection of the set $\aaa_{\xi}^{\cc}$ defined by \eqref{u0xi} with $\La\o(n,\C)$,
then $\aaa_{\xi}^{\rr}$ is the $\u_{\xi}^0$
of \cite[Proposition 2.7]{burstall-guest} for $\g=\o(n)$, and the exponential map sends
$\aaa_{\xi}^{\rr}$ to $\A_{\xi}^{\rr}$.

By \cite[p.\ 560]{burstall-guest} there is an associated extended solution $\check\Phi:M \to \Omega(\SO(n)/C)$,  canonical element
$\check{\xi} \in \o(n)$ and meromorphic map $B:M \to \aaa_{\xi}^{\rr}$ such that, setting
$A = \exp B$,  
$\check\Psi = A \ga_{\check{\xi}}:M \to \La_{\alg}(\O(n,\C)/C)$ is a meromorphic complex extended solution with $[\check\Psi] = \check\Phi$; explicitly, there is a loop $\eta \in \Omega\O(n)/C$ such that
$\la^{-\wt r/2}\pi\circ\check\Phi = \eta [\check{A}\ga_{\check{\xi}}]$.
   
Set $r = 2\check\xi_n$, then $\ga_{\xi} = \la^{r/2}\ga_{\check\xi}$ is a canonical element in 
$\Omega_r\U(n)^{\rr}$ and $A\ga_{\xi}:M \to \La_{\alg}\O(n,\C)$ is a complex extended solution.
Set $\Phi = [A\ga_{\xi}]:M \to \Omega_r\U(n)^{\rr}$.  Then $\wt\Phi\Phi^{-1}:M \to \Omega_{\alg}\U(n)$ satisfies
$\pi\circ(\wt\Phi\Phi^{-1}) = \eta\la^{(\wt r-r)/2}$, which is independent of $z \in M$.  Hence
$\wt\Phi\Phi^{-1}$ is also independent of $z$, i.e., is a loop in $\U(n)$, so that $\Phi$ is equivalent to $\wt\Phi$ and we are done.
\end{proof}

\begin{remark} \label{rem:parity}
(i) Let $f$ be a holomorphic map $M \to \O(2m)/\U(m)$ \ $(m >1)$ so that $f$ is a maximally isotropic
holomorphic subbundle of $\CC^{2m}$, then $f$ has polynomial associated extended solution
$\wt\Psi = \pi_{f} + \la\,\pi_{f}^{\perp}:M \to \Omega_1\U(2m)^{\rr}$.
The above proof constructs the extended solution
 $\check{\Phi} =  \la^{-1/2}\wt\Psi  = (1/\la^{1/2})\pi_{f} + \la^{1/2}\,\pi_{f}^{\perp}:M \to \Omega(\SO(2m)/C)$ which can be written in the form $[A\ga_{\check{\xi}}]$ with 
$\check{\xi} = \ii\diag(1/2,\ldots,1/2,-1/2,\ldots,-1/2)$, a canonical element of $\o(2m)$, and $A$ as in \eqref{A-r1}.  Then $\Psi = \la^{1/2}\check{\Phi} = \wt\Psi$ is of the form $[A\ga_{\xi}]$ with $\xi=\ii\diag(1,\ldots,1,0,\ldots,0)$,
a canonical element of $\Omega_1\U(2m)^{\rr}$, and $A:M\to \A^{\rr}_{\xi}$.

(ii) Given $\xi$ and $A \in \Sol_{\xi}^{\rr}$, we can find the extended solution $\Phi = [A\ga_{\xi}]$ and the resulting harmonic map $\Phi_{-1}$ explicitly from $W = A\ga_{\xi}\H_+$ as a product of unitons, by using the alternating factorization
\cite[\S 6.1]{unitons} given by \eqref{alt-fact}, \eqref{alpha} and \eqref{fact}.
In the $S^1$-invariant case, we have the simpler procedure: set $\al_i = $ the span of the columns $c_j$ of $A$ with $\xi_j<i$; then the corresponding factors $\pi_{\al_i} + \la \pi_{\al_i}^{\perp}$ are unitons which commute and give the Segal, Uhlenbeck and alternating factorizations depending on the order in which they are written.
\end{remark}

\subsection{Adding a border to increase dimension} \label{subsec:add-border}

We will give a method of finding parametrizations of complex extended solutions of finite uniton number from a Riemann surface $M$ to $\O(n)$ by induction on the dimension $n$.  Our starting point is Proposition \ref{prop:BuGu-On} which reduces the problem to finding, for each canonical element $\xi$, all meromorphic maps $A:M \to \A_{\xi}^{\rr}$ satisfying the extended solution equation \eqref{GrM}.
We shall give an algorithm for parametrizing such $A$.

So let $\xi = \ii\diag(\xi_1,\xi_2, \ldots,\xi_{n-1},\xi_n)$
be a canonical element of $\Omega_r\U(n)^{\rr}$ for some $r \in \N$, $n \geq 3$;
denote its type by $(t_0,t_1,\ldots, t_{r-1}, t_r)$.
Set $\wt\xi = \ii\diag(\xi_2,\ldots, \xi_{n-1})$. 
Then, unless it has type $(1,1)$, $\wt\xi$ is a canonical element of
$\Omega_{\wt r}\U(n-2)^{\rr}$  
whose type $(\wt t_0, \ldots, \wt t_{\wt r})$ is equal to
$(t_0\!-\!1,t_1,\ldots, t_{r-1}, t_r\!-\!1)$ with $\wt r = r$ if
$t_0\ ( =t_r) \geq 2$, and $(t_1,\ldots, t_{r-1})$ with
$\wt r = r-2$ otherwise.  If $\wt\xi$ has type $(1,1)$ and so is not canonical, then $n=4$ and $\xi$ is of type $(2,2)$; we will treat that case separately.

Given $A:M \to \Omega\O(n,\C)$ with values in
$\A_{\xi}^{\rr}$,
the matrix $\wt{A}$ obtained by \emph{removing the border}, i.e., $\wt{A} = (a_{ij})_{i,j=2,\ldots,n-1}$ defines a map from $M$ to $\A_{\wt\xi}^{\rr}$.
Conversely, given $\wt A = (a_{ij})_{i,j=2,\ldots, n-1}:M \to \A_{\wt\xi}^{\rr}$ 
we define a map $A:M \to \A_{\xi}^{\rr}$ by a process of \emph{adding a border}. 
This consists of adding a \emph{new top row} \ 
$(a_{12}, \ldots, a_{1,n-1})$, \emph{new last column}
$(a_{2n}, \ldots, a_{n-1,n})^{\T}$ and \emph{new top-right element} $a_{1n}$, and then completing the border by  setting $a_{i1}=\delta_{i1}$ and $a_{nj} = \delta_{nj}$ for $i,j=1,\ldots,n$. 
Note that our definitions of `new top row' and `new last column' exclude the new top-right element $a_{1n}$. Note also that, given $\wt{A}$ and either the new top row or the new last column, we can find the rest of the matrix by imposing the complex-orthogonality \eqref{cx-orthog} of the columns $c_i$ of $A$; in fact, using $(c_i,c_n)=0$ for $i=2,\ldots, n-1$ in turn gives the new top row from the new last column or vice-versa, and then using $(c_n,c_n) = 0$ gives the new top-right element.  We refer to this as completing the matrix \emph{by algebra}.   Note that, although removing the border preserves symmetry and $S^1$-invariance (by Remark \ref{rem:A}(v)), adding a border may destroy these, depending on the data chosen.

The following lemma underpins the induction step.
For a canonical element $\xi$ of type $(t_0,\ldots, t_r)$, define integers $0 = T_{r+1} < T_r < \cdots < T_0 = n$ by $T_k = \sum_{j=k}^r t_j$.
Note that $\xi_i=k$ precisely when $T_{k+1} < i \leq T_k$.

\begin{lemma} \label{lem:algebra}
Let $\xi$ be a canonical element of\/ $\Omega_r\U(n)^{\rr}$
(Definition \ref{def:can-el-R}) not of type $(2,2)$, and let $A = (a_{ij})_{i,j=1,\ldots, n}:M \to \A_{\xi}^{\rr} \subset \Omega\O(n,\C)$ be holomorphic.
Define $\wt A:M \to \Omega\O(n-2,\C)$
by
$\wt A = (a_{ij})_{i,j = 2, \ldots, n-1}$.  Then
$\wt A$ is holomorphic and has values in $\A_{\wt\xi}^{\rr}$ for the canonical element $\wt\xi$ obtained from $\xi$ as above.

\begin{enumerate}
\item[(i)] Suppose that $A:M \to \A_{\xi}^{\rr}$ satisfies the extended solution equation \eqref{GrM}.  Then so does $\wt{A}:M \to \A_{\wt\xi}^{\rr}$\,.

\item[(ii)] Conversely, suppose that
$\wt A:M \to \A_{\wt\xi}^{\rr}$ satisfies \eqref{GrM}.  Then the following are equivalent$:$
\begin{enumerate}
\item[(a)]  $A:M \to \A_{\xi}^{\rr}$ satisfies
\eqref{GrM}$;$

\item[(b)] the entries in the new top row satisfy \eqref{GrM1}, i.e., 
\be{GrM1j}
a_{1j}' = \sum_{i:\,\xi_i > \xi_j}\la^{\xi_i-\xi_j-1}\rho_{ij}'a_{1i} \quad \mod \la^{r-\xi_j-1} \qquad (j=T_r,\ldots, n-1);
\ee

\item[(c)] the entries of the new last column satisfy \eqref{GrM1}, i.e.,
\be{GrMin}
a_{in}' = \sum_{j:\, \xi_i \geq \xi_j > 0}\la^{\xi_j-1}\rho_{jn}'a_{ij} \quad \mod \la^{\xi_i-1}
\qquad (i= 2,\ldots, T_1).
\ee
\end{enumerate}
\end{enumerate}
\end{lemma}

\begin{proof}
\emph{First, suppose that {\rm (c)} holds}.  As usual, let $c_1,\ldots, c_n$ denote the columns of $A$; let $\wc_1,\ldots, \wc_n$ denote the same columns omitting top and bottom entries, i.e.,
$\wt c_j = (a_{2j}, \ldots, a_{n-1,j})^{\T}$ (note that $\wt c_1$ has all entries zero).
As in Proposition \ref{prop:GrM}, hypothesis (c) is equivalent to
\be{(c)}
\wc_n' = \sum_{j:\ \xi_j >0} \la^{\xi_j-1}\rho_{jn}'\wc_j\,.
\ee
By complex-orthogonality \eqref{cx-orthog},  $(c_j,c_n) = 0$ \ $(j =2,\ldots, n)$. 
Expanding this gives
$a_{1j} + (\wc_j,\wc_n) = 0$, then 
differentiating the last equation gives
\be{a1j}
a_{1j}' = -(\wc_j',\wc_n) - (\wc_j,\wc_n')\,.
\ee
By \eqref{(c)}, the second term on the right-hand side of
\eqref{a1j} is
$(\wc_j,\wc_n')
= \sum_{i:\, \xi_i >0}\la^{\xi_i-1} \rho_{in}'(\wc_j,\wc_i)
= \la^{\xi_{\jbar}-1} \rho_{\jbar n}' = 0 \mod \la^{r-\xi_j-1}$
using complex-orthogonality for $\wt{A}$ and $\xi_{\jbar} = r-\xi_j$.

As for the first term on the right-hand side of \eqref{a1j}, by the extended solution equation \eqref{GrM} for $\wt{A}$, we have $\wc_j' = \sum_{i \geq 2: \xi_i > \xi_j} \la^{\xi_i-\xi_j-1}\rho_{ij}' \wc_i$ so that
$(\wc_j',\wc_n) = \sum_{i \geq 2: \xi_i > \xi_j} \la^{\xi_i-\xi_j-1}\rho_{ij}'(\wc_i,\wc_n)$.
Now $(\wc_i,\wc_n) + a_{1i} = (c_i,c_n)$ which is zero for $i \geq 2$ by \eqref{cx-orthog}. Hence,
$(\wc_j', \wc_n) = -\sum_{i \geq 2: \xi_i > \xi_j} \la^{\xi_i-\xi_j-1} \rho_{ij}'a_{1i}$, and
then, adding in the second term calculated above,
\eqref{a1j} gives
$a_{1j}' =  \sum_{i \geq 2: \xi_i > \xi_j} \la^{\xi_i-\xi_j-1}\rho_{ij}'a_{1i}$ $\mod \la^{r-\xi_{j}-1}$, which is equivalent to (b) by Remark \ref{rem:GrM}(i).

\medskip

We also see that \eqref{GrM} holds for the top-right entry, indeed,
expanding $(c_n,c_n) = 0$ gives 
$a_{1n} = -\half ( \wc_n, \wc_n)$.  Differentiating this and using
\eqref{(c)} gives
$a_{1n}' = - \sum_{i:\xi_i > 0} \la^{\xi_i-1}\rho_{in}' (\wc_i,\wc_n) = \sum_{i:\xi_i > 0} \la^{\xi_i-1}\rho_{in}' a_{1i}$.
So \emph{{\rm (c)} implies that \eqref{GrM} holds for all columns of $A$ including the last, i.e., {\rm (a)} holds}.

\medskip

\emph{Next, assume that {\rm (b)} holds}.  We prove that (c) holds by downward induction on $i \in [2,T_1]$.  
For $T_2 < i \leq T_1$ so that $\xi_i = 1$, \eqref{GrMin} is trivially true as it says
$a_{in}' = \rho_{in}'$.
We may thus use $I=T_2+1$ as the starting point of our induction.

We now use the notations $\hat{c}_k^i = (a_{1 k}, \ldots, a_{i-1,k})^{\T}$ for the part of $c_k$ `above' $a_{ik}$ and
$\check{c}_k^{\ibar} = (a_{\ibar+1,k}, \ldots, a_{nk})^{\T}$ for the part of $c_k$ `below' $a_{\ibar k}$\,; note these are both columns of length $i-1$.
Suppose \eqref{GrMin} holds for $i > I$ for some
$I \in \{2,\ldots, T_2\}$.   We show that it holds for $i=I$,
i.e., that
\be{aIn'}
a_{In}' = \sum_{j:\,\xi_I \geq \xi_j >0}\la^{\xi_j-1}\rho_{jn}'a_{Ij}
\mod \la^{\xi_I-1}. 
\ee
Clearly, $a_{In} + (\hat{c}_{\Ibar}^{\Ibar}, \check{c}_n^I) =(c_{\Ibar},c_n)$ which is zero since $\Ibar > 1$.   Differentiating this gives
\be{ain'}
(a_{In})' = - \bigl((\hat{c}_{\Ibar}^{\Ibar})', \check{c}_n^I \bigr)
	- \bigl(\hat{c}_{\Ibar}^{\Ibar}, (\check{c}_n^I)'\bigr).
\ee
By \eqref{GrM}, the first term on the right-hand side of \eqref{ain'} is
$\bigl((\hat{c}_{\Ibar}^{\Ibar})', \check{c}_n^I \bigr)
	= \sum_{j:\,\xi_j > \xi_{\Ibar}} \la^{\xi_j-\xi_{\Ibar}-1}\rho_{j\Ibar}'(\hat{c}_j^{\Ibar}, \check{c}_n^I)$.
But $(\hat{c}_j^{\Ibar}, \check{c}_n^I) = (c_j,c_n) = \delta_{\jbar n}$ by \eqref{cx-orthog}, so that
$\bigl((\hat{c}_{\Ibar}^{\Ibar})', \check{c}_n^I \bigr)
	= \la^{\xi_I-1}\rho'_{1\Ibar}$\,.
 
By the induction hypothesis, the second term on the right-hand side of \eqref{ain'} is
\be{second}
\bigl(\hat{c}_I^I, (\check{c}_n^I)'\bigr)
	= \sum_{j:\, \xi_j >0} \la^{\xi_j-1}\rho_{jn}' \bigl(\hat{c}_{\Ibar}^{\Ibar},
\check{c}_j^I\bigr).
\ee
We show the general term in the sum on the right-hand side of \eqref{second} is given by
\be{general}
\la^{\xi_j-1} \rho_{jn}'
		\bigl(\hat{c}_{\Ibar}^{\Ibar}, \check{c}_j^I\bigr)
= -\la^{\xi_j-1} \rho_{jn}' a_{Ij} \mod \la^{\xi_I-1}.
\ee	
First, $(\hat{c}_{\Ibar}^{\Ibar}, \check{c}_j^I) + a_{Ij} = (c_{\Ibar}, c_j)$,
which is zero for $j \neq I$ by \eqref{cx-orthog}, so \eqref{general} holds for this case.
On the other hand, if $j = I$, then the left-hand side of \eqref{general} is zero since
$\check{c}_I^I$ is a zero column and the right-hand side is a multiple of $\la^{\xi_I-1}$, so the two sides are equal $\mod \la^{\xi_I-1}$ as required.

Substituting \eqref{general} into \eqref{second} and then into \eqref{ain'} we obtain \eqref{GrMin} for $i=I$ completing the induction step, and so (c) holds.  This completes the proof of the lemma.
\end{proof}

\subsection{Parametrization of extended solutions for $\O(n)$}
\label{subsec:param}

By a \emph{generalized derivative} of a meromorphic function $\nu$
on $M$ we mean a quotient $\nu'/e$, where ${}'$ denotes derivative with respect to some local complex coordinate $z$ on $M$ and
$e = \sum \beta'_j b_j$ is a finite sum which is not identically zero.
Here $\beta_j$ and $b_j$ are meromorphic functions on  $M$; note that the 
quotient $\nu'/e$ is independent of the choice of complex coordinate $z$ on $M$.  In particular, we shall call a generalized derivative of the form
$\nu'/\beta'$ with $\beta$ meromorphic and non-constant the \emph{generalized derivative of\/ $\nu$ with respect to $\beta$}; all generalized derivatives are locally of this form.  Away from points where $\beta$ has a pole or $\beta'$ is zero, $\beta$ gives an alternative complex coordinate to $z$ and
$\nu'/\beta'$ is the derivative of $\nu$ with respect to that complex coordinate.
When the denominator is unimportant, we shall often denote a generalized derivative by $\nu^{(1)}$ and higher generalized derivatives by $\nu^{(2)},\nu^{(3)}, \ldots$, and we set $\nu^{(0)} = \nu$;
 thus for any $d \geq 1$,
$\nu^{(d)}$ is the generalized derivative of $\nu^{(d-1)}$ given 
by
\be{nu}
\nu^{(d)} = (\nu^{(d-1)})'/e_{d-1}
\ee
where $e_{d-1} = \sum_j \beta'_{d-1,j} b_{d-1,j}$ is a finite
 sum with $\beta_{d-1,j}$ and $b_{d-1,j}$ meromorphic on $M$.
For example, if $\nu^{(1)} = \nu'/e_0$ then
$\nu^{(2)} = (\nu^{(1)})'/e_1 = (\nu'/e_0)'/e_1$.

Let $\M(M)$ denote the space of meromorphic functions on the surface.  Let $\xi$ be a canonical element of $\Omega_r\U(n)^{\rr}$ for some $r$, set
$p = p(\xi) = \sum_{k=1}^r \dim \g_k^{\rr}(\xi)$ and $p_1 = p_1(\xi) =  \dim\g_1^{\rr}(\xi)$. 
Recall the spaces $\Sol_{\xi}^{\rr}$ and $(\Sol_{\xi}^{\rr})_0$ from
Corollary \ref{cor:one-one}. 
 
\begin{proposition} \label{prop:alg} The algorithm below defines a mapping $h = h_{\xi}:\M(M)^p \to \Sol_{\xi}^{\rr}$.
It restricts to an algorithm which defines a mapping
$h_0 = (h_0)_{\xi}: \M(M)^{p_1} \to (\Sol_{\xi}^{\rr})_0$. 
\end{proposition}

{\it Proof.}
\emph{We first give the algorithm which defines $h_0:\M(M)^{p_1} \to (\Sol_{\xi}^{\rr})_0$
for any $\xi$.}   
This is trivial when $n=1,2$, as $\O(n,\C) = \{I\}$ so $\xi=\ii I$ and $\Sol_{\xi}^{\rr} = (\Sol_{\xi}^{\rr})_0 = \{I\}$.
We use these as a base for an induction on the dimension $n$: in the induction step $n$ is increased by $2$.

Let $n \geq 3$.  Given a canonical element $\xi=\ii\diag(\xi_1,\ldots,\xi_n)$ of $\Omega_r\U(n)^{\rr}$ 
define a canonical element $\wt\xi =\ii\diag(\wt\xi_1,\ldots, \wt\xi_{n-2})$ of $\Omega_{\wt r}\U(n-2)^{\rr}$ as in \S \ref{subsec:add-border}.
As induction hypothesis, \emph{suppose that we have determined 
$\wt h_0:\M(M)^{\wt p_1} \to (\Sol_{\wt\xi}^{\rr})_0$}. 
We show how to find $h_0:\M(M)^{p_1} \to (\Sol_{\xi}^{\rr})_0$
explicitly from $\wt h_0$.
Recall that all $A:M \to (\A_{\xi}^{\rr})_0$ in $(\Sol_{\xi}^{\rr})_0$ are obtained from some
$\wt{A}:M \to (\A_{\wt\xi}^{\rr})_0$ in $(\Sol_{\wt\xi}^{\rr})_0$ by adding a border as in \S \ref{subsec:add-border}.  We shall find a parametrization of the new first row
$(a_{12}, \ldots, a_{1,n-1})$ by solving the equation
\eqref{GrM1j} which now reads
\be{GrM1k}
a_{1k}' = \sum_{j:\,\xi_j = \xi_k+1}\rho_{jk}'a_{1j}
\quad  \qquad (k=t_r+1,\ldots, n-1).
\ee
Now $a_{1k} = \delta_{1k}$ when $\xi_k =r$, i.e., 
for $1 \leq k \leq t_r\ (= t_0)$. 
The next $t_{r-1}\ (=t_1)$ entries $\{a_{1k} : t_r+1  \leq k \leq T_{r-1}\}$
(where $T_{r-1} = t_{r-1}+t_r$) give the entries of $\g_1^{\rr}(\xi)$ which are not in $\g_1^{\rr}(\wt{\xi})$, thus
$t_{r-1} = t_1 = \dim\g_1^{\rr}(\xi) -\dim\g_1^{\rr}(\wt\xi)$. 
These entries $a_{1k}$ have no equation to satisfy: \eqref{GrM1k} holds identically for them. 

To find our parametrization, we initially parametrize the above entries by meromorphic functions $\nu = (\nu_{0,1},\ldots, \nu_{0,t_{r-1}})$, setting $a_{1,t_r+i} = \nu_{0,i}$ \ $(i=1,\ldots, t_{r-1})$.  These are essentially the parameters used in \cite{burstall-guest}, however, they will not usually be our final choice of parameters.
For the next entry, \eqref{GrM1k} reads 
\be{start}
a_{1k}' = \sum_{j=t_r+1}^{T_{r-1}} \rho_{jk}' a_{1j} =  \sum_{i=1}^{t_{r-1}} \rho_{i+t_r,k}' \nu_{0,i}
	\quad (k = T_{r-1}+1)\,.
\ee
By the inductive hypothesis, the $\rho_{jk}$ are known functions of the parameters $\mu$ for $\wt{A}$. We now replace our initial choice of parameters $\nu_{0,i}$ by a new choice $\nu_{1,i}$ of parameters where the `old' parameters $\nu_{0,i}$ are given in terms of the new ones by 
$\nu_{0,i} = \nu_{1,i}$ if $\rho_{i+t_r,k}'$ is identically zero, and 
$\nu_{0,i} =$ the generalized derivative $\nu_{1,i}^{(1)}:=(\nu_{1,i})'/\rho_{i+t_r,k}'$, otherwise.
Then integrating \eqref{start} gives
$$
a_{1k} = \sum_{i=1}^{t_{r-1}}b_{ik}(\mu)\nu_{1,i} 
\qquad (k = T_{r-1}+1) \,
$$
where $b_{ik}(\mu) = 0$ when $\rho_{i+t_r,k}'$ is identically zero, and $b_{ik}(\mu) = 1$ otherwise; thus the value of $b_{ik}(\mu)$ depends on $\mu$. Note that the previous entries
$a_{1k}$ can now be written in terms of the new parameters,
in fact,
\be{ind-hyp1}
a_{1k} = \nu_{0,1} = \sum_{i=1}^{t_{r-1}}
\bigl\{ b_{ik0}(\mu)\nu_{1,i} + b_{ik1}(\mu)\nu_{1,i}^{(1)}
\bigr\} \qquad (k =t_r+1,\ldots, T_{r-1})
\ee
for some functions $b_{ikp}(\mu)$ (which are here just $0$ or $1$).

We prove by induction that,
\emph{for each $K = 1,\ldots, n-1-T_{r-1}$, there are parameters
$\nu = (\nu_{K,1},\ldots,\nu_{K,t_{r-1}})$ with each $\nu_{K\!\!-1,i}$ equal either to $\nu_{K,i}$, or to a generalized derivative $\nu_{K,i}^{(p)}$ of
$\nu_{K,i}$ with respect to a function of $\mu$, such that} 
\be{ind-hypK}
a_{1k} = \sum_{i=1}^{t_{r-1}} \sum_{p=0}^{K\!\!-1}  b_{ikp}(\mu)
\nu_{K,i}^{(p)}
= \sum_{i=1}^{t_{r-1}}  b_{ik0}(\mu)\nu_{K,i} 
+ \sum_{i=1}^{t_{r-1}} \sum_{p=1}^{K\!\!-1}  b_{ikp}(\mu)
\nu_{K,i}^{(p)}
\quad (T_{r-1}+1 \leq k \leq T_{r-1}+K).
\ee
Here each $b_{ikp}$ is now a rational function of the parameters $\mu$ for $\wt{A}$ and the derivatives of those parameters,
and $\{\ \ \}^{(p)}$ denotes a $p$th generalized derivative as explained above.
This is established for $K=1$ by \eqref{ind-hyp1}.
 
Suppose we know that, \emph{for some $K$ with $2 \leq K \leq n-1-T_{r-1}$, \eqref{ind-hypK} holds with $K$ replaced by
$K\!\!-1$, i.e.},
\be{ind-hypK-1}
a_{1k} = \sum_{i=1}^{t_{r-1}} \sum_{p=0}^{K-2}{\wt b}_{ikp}(\mu)
\nu_{K\!\!-1,i}^{(p)} \qquad (T_{r-1}+1 \leq k \leq T_{r-1}+K-1);
\ee
then we shall deduce that \eqref{ind-hypK} holds.
{}From \eqref{GrM1k}  we have
$a_{1K}' = \sum_{j=1}^{K\!\!-1} \rho_{jK}' a_{1j}$.  Using the induction hypothesis \eqref{ind-hypK-1} for each $a_{1j}$ gives us
$$
a'_{1K} = \sum_{i=1}^{t_{r-1}} \Bigl\{ \sum_{j=1}^{K\!\!-1} \rho_{jK}'
\wt{b}_{ij0} \nu_{K\!\!-1,i} + \sum_{p=1}^{K-2} c_{ip} \nu_{K\!\!-1,i}^{(p)} \Bigr\}
\quad \text{where} \quad
c_{ip} = \sum_{j=1}^{K\!\!-1} \rho'_{jK} \wt{b}_{ijp}\,.
$$
We integrate by parts each term in the last sum, first interpreting
$\nu_{K\!\!-1,i}^{(p)}$ using \eqref{nu}, as follows:
$$
\int\! c_{ip} \nu_{K\!\!-1,i}^{(p)} = \int\! \wt{c}_{ip} (\nu_{K\!\!-1,i}^{(p-1)})'
= \wt{c}_{ip}\nu_{K\!\!-1,i}^{(p-1)} - \int\! \wt{c}_{ip}' \nu_{K\!\!-1,i}^{(p-1)}
$$
for some functions $\wt{c}_{ip}(\mu)$.  Repeating the procedure $p$ times gives
\be{a1K}
a_{1K} = \sum_{i=1}^{t_{r-1}} \Bigl\{ \sum_{p=0}^{K-2} d_{ip}
\nu_{K\!\!-1,i}^{(p)} +  \int\! e_i \nu_{K\!\!-1,i} \Bigr\} + c_{K}
\ee
for some functions $d_{ip}(\mu), f_i(\mu)$ and constant of integration $c_K$.
Here, for each $i=1,\ldots, t_{r-1}$, $e_i = \sum_{j=1}^{K\!\!-1}\rho_{jK}'\wt{b}_{ij0}
+ f_i'$\,.

We now replace the parameters $\nu_{K\!\!-1,i}$ by `new' parameters $\nu_{K,i}$ where the `old' parameters $\nu_{K\!\!-1,i}$ are given in terms of the new ones as follows.
If $e_i$ is identically zero, $\nu_{K\!\!-1,i} = \nu_{K,i}$; we call this a \emph{degenerate step} and say that the algorithm is \emph{degenerate} if this ever occurs.
  Otherwise, $\nu_{K\!\!-1,i}$ is equal to the generalized derivative $(\nu_{K,i})'/e_i$, so that the integral in \eqref{a1K} evaluates to $\nu_{K,i}$.
  If not \emph{all} $e_i$ are identically zero, we may absorb the constant $c_K$ of integration into one of the new parameters $\nu_{K,i}$; however, if all $e_i$ are identically zero, then we cannot.
In this case, we remove $c_K$ by premultiplying $A$ by a matrix $E = (e_{ij}) \in \O(n,\C)$ which is the identity matrix except that $e_{1K}=-c$, $e_{n+1-K,n} = c$ and, if $n$ is odd and $K = (n+1)/2$,  $e_{1n} = -\half c^2$.  This does not alter $\wt{A}$ or any previous entries $a_{1k}$ \ $(k < K)$ of the new first row.
   This establishes \eqref{ind-hypK} for $k=K$.  

Finally, for $k < K$ we replace the $\nu_{K\!\!-1,i}$ in \eqref{ind-hypK-1} by the expressions in terms of $\nu_{K,i}$ just given, and the induction step is complete.
This gives the new first row $(a_{12}, \ldots, a_{1,n-1})$;
we complete the matrix finding the new last column
$(a_{2n}, \ldots, a_{n-1,n})^{\T}$ and new top-right element $a_{1n}$ by algebra, i.e.,
imposing that $A$ has values in $\O(n,\C)$ by using \eqref{cx-orthog}, see \S \ref{subsec:add-border}.
We have now given an algorithm for finding $h_0:\M(M)^{p_1} \to (\Sol_{\xi}^{\rr})_0$ from $\wt h_0$ which completes the induction on dimension.

Note that the subset of data where the algorithm is degenerate at some stage in the induction forms an algebraic subvariety of $\M(M)^{p_1}$; define $\M(M)^{p_1}_{\wt{\ND}}$ to be its complement. 

\emph{We now extend the algorithm to define a map
$h:\M(M)^p \to \Sol_{\xi}^{\rr}$.}
We follow the same method of adding a border, then the equations to satisfy for the first row are again \eqref{GrM1j} but now each element is a polynomial in $\la$; we write $a_{ij}^q$ for the coefficient of $\la^q$ in $a_{ij}$.
When $i=1$, for each $j$, $a_{1j}$ is a polynomial of degree at most
$\xi_i-\xi_j - 1 = r-\xi_j -1$.
We now equate coefficients of $\la^q$ in equation \eqref{GrM1k}.
For the highest possible degree on the left-hand side, $q = r-\xi_k -1$, there is no equation to satisfy since we are working $\!\!\mod \la^{r-\xi_k-1}$.
Thus our initial choice of data for the first row will be $\{a_{1k}^{r-\xi_k-1} : t_r+1 \leq k \leq n-1\}$;
note that this does not include $a_{1n}$, which is determined by algebra, see \S \ref{subsec:add-border}.
We set $a_{1k}^{r-\xi_k-1} = \nu^0_{k-t_r}$ \ $(k=t_r+1,\ldots, n-1)$ giving our initial choice of parameters $\nu^0 = (\nu^0_1,\ldots, \nu^0_{n-1-t_r})$.

For $q < r-\xi_k-1$, by equating coefficients of $\la^q$ we obtain the equations:
\be{la-q}
(a_{1k}^q)' = \sum_{j:\,0 \leq \xi_j - \xi_k - 1 \leq q}
\rho'_{jk} a_{1j}^{q-(\xi_j-\xi_k-1)}
\qquad (T_{r-1}+1 \leq k \leq n-1,\  0 \leq q < r-\xi_k-1).
\ee
Note that the sum is over the $q+1$ blocks preceding that containing $a_{1k}$: since $q+\xi_k+1 < r$, this never includes the entries $\{a_{1j}:\xi_j = r\}$ in the left-most block.  
Note also that, for each $j$ the sum concerns the coefficient of $\la^{q-(\xi_j-\xi_k-1)}$ of $a_{1j}$;
since $q < r-\xi_k-1$, this is at most $r-\xi_j -1$, the maximum possible power for $a_{1j}$.
Finally note that the condition $\xi_k \leq r-q-1$ is saying that $a_{1k}$ is in the block where
$\xi_k = r-q-1$ or in a block to the right of that.
For clarity, we write out the first three equations of \eqref{la-q}: 
\begin{eqnarray*}
(a_{1k}^0)' &=& \sum_{j:\,\xi_j-\xi_k-1=0} \rho'_{jk} a_{1j}^0 \qquad (k \leq n-1,\,\xi_k \leq r-1), \\
(a_{1k}^1)' &=&  \sum_{j:\,\xi_j-\xi_k-1=0} \rho'_{jk} a_{1j}^1
			\!\! + \sum_{j:\,\xi_j-\xi_k-1=1} \rho'_{jk} a_{1j}^0
									\qquad (k \leq n-1,\,\xi_k \leq r-2),\\
(a_{1k}^2)' &=& \sum_{j:\,\xi_j-\xi_k-1=0} \rho'_{jk} a_{1j}^2
+ \!\! \sum_{j:\,\xi_j-\xi_k-1=1} \rho'_{jk} a_{1j}^1 + \!\!
\sum_{j:\,\xi_j-\xi_k-1=2} \rho'_{jk} a_{1j}^0
										\qquad \! (k \leq n-1,\,\xi_k \leq r-3).
\end{eqnarray*}

We solve \eqref{la-q} for each $k$ by induction on $q$ with initial data $\nu^0$ as above; we omit the details.

Putting the initial data for each new first row together shows that our \emph{initial data} for finding
$A \in \Sol_{\xi}^{\R}$ is $\{a_{ij}^{\xi_i-\xi_j-1} : \xi_i > \xi_j\,, j < \ibar\, \}$, i.e., the $\la^q$-coefficient of each entry of $A$ on the part of the $(q+1)$st block superdiagonal of $A$ above the second diagonal, for $q=0, 1,2,\ldots, r-1$. 
Note that this initial data is related to that in \cite{burstall-guest} by the exponential map; it is, however, our \emph{final data} which forms $\mu \in \M(M)^p$.   

Again, the subset of (final) data where the algorithm is degenerate at some stage in the induction forms an algebraic subvariety of $\M(M)^p$; define $\M(M)^p_{\wt{\ND}}$ to be its complement.
\qed

\medskip

We now see how the above algorithm gives parametrizations of extended solutions of canonical type: recall that by Proposition \ref{prop:BuGu-On}, any extended solution is equivalent to one of canonical type.

\begin{theorem} \label{th:gen}
Let $M$ be a Riemann surface.  Let\/ $\xi=\ii\diag(\xi_1,\ldots,\xi_n)$ be a canonical element of\/ $\Omega_r\U(n)^{\rr}$ for some $r \in \N;$
set $p = p(\xi) = \sum_{k=1}^r \dim \g_k^{\rr}(\xi)$ and
$p_1 = p_1(\xi) = \dim \g_1^{\rr}(\xi)$.   Let
$h=h_{\xi}:\M(M)^p \to \Sol_{\xi}^{\rr}$ and $h_0 = (h_0)_{\xi}:\M(M)^{p_1} \to (\Sol_{\xi}^{\rr})_0$ be the mappings of Proposition \ref{prop:alg}.
\begin{enumerate}
\item[(i)] The maps $h$ and $h_0$ are locally surjective up to replacing $A$ by $EA$ for some constant matrix
$E \in \O(n,\C)$.

\item[(ii)] The map $h$ restricts to a locally surjective mapping $h:\M(M)^p_{\wt{\ND}} \to (\Sol_{\xi}^{\rr})_{\wt{\ND}}$, $\mu \mapsto A(\mu)$ to an open dense subset of\/ $\Sol_{\xi}^{\rr}$; this map
is algebraic in the sense that each entry of\/ $A(\mu)$ is polynomial
in $\la$ with coefficients rational functions of the $\mu_i$ and their derivatives of order less than or equal to $n-3$.  
  
The extended solution $\Phi = [A\ga_{\xi}]:M \to \Omega_r\U(n)^{\rr}$ corresponding to a choice of
$\mu \in \M(M)^p$  is
given explicitly by \eqref{alt-fact}, \eqref{alpha} and \eqref{fact}$;$ each entry of\/ $\Phi$ is
polynomial in $\la$ with coefficients rational functions of the $\mu_i$, their derivatives of order less than or equal to $n-3$ and the complex conjugates of these.

\item[(iii)] The map $h_0$ restricts to a locally surjective mapping
$h_0:\M(M)^{p_1}_{\wt{\ND}} \to (\Sol_{\xi}^{\rr})_0^{\wt{\ND}}$ to an open dense subset of\/ $(\Sol_{\xi}^{\rr})_0$;
this map is algebraic in the sense that each entry of\/ $A(\mu)$ is a rational function of the $\mu_i$ and their derivatives of order less than or equal to $n-3$.

In this case, the extended solution $\Phi = [A\ga_{\xi}]:M \to \Omega_r\U(n)^{\rr}$ corresponding to a choice of 
$\mu \in \M(M)^{p_1}$ is $S^1$-invariant and is given explicitly by \eqref{fact} where
$\al_i$ is the span of columns $c_j$ of\/ $A$ with $\xi_j < i$.
\end{enumerate}
 \end{theorem}

Note that, since the value of $h(\mu)$ at a point of $M$ only depends on the germ of $\mu$ at that point, $h$ restricts to a map on $\M(U)^p$ for any open subset $U$ of $M$.
To say that $h$ is \emph{locally surjective} means that, given an extended solution $\Phi:M \to \Omega_r\U(n)^{\rr}$ of type $\xi$, there is a discrete set $D$ of points such that, for any point of $M \setminus D$, there is an open neighbourhood $U$ of $p$, such that $\Phi|_U = h(\mu)$ for some $\mu \in \M(U)^p$ defined on $U$. Similarly for $h_0$.

\begin{proof} (i)  Given a solution, we can read off the values of the \emph{initial} parameters which give it from its entries, viz.
$\{a_{ij}^{\xi_i-\xi_j-1} : \xi_i > \xi_j\,, j < \ibar\, \}$.  
The values of our final parameters $\mu_i$ can then be found from those initial parameters by a finite number of integrations, and premultiplication by constant matrices in $\O(n,\C)$ in the degenerate case.  As we can only integrate on a simply-connected open set and must avoid the discrete set of simple poles for every integration, this shows local surjectivity. 

(ii) Let $(\Sol_{\xi}^{\rr})_{\wt{\ND}}$ be the subset of solutions in $\Sol_{\xi}^{\rr}$ given locally by the algorithm with no degenerate steps.  The rest is clear; similarly for (iii).
\end{proof}

Since, by Proposition \eqref{prop:BuGu-On}, every harmonic map has an associated extended solution $\Phi$ in $\Sol_{\xi}^{\rr}$, and the formulae for $\Phi$ given by \eqref{alpha}, \eqref{alt-fact} and
\eqref{fact} introduce complex conjugates, we deduce 

\begin{corollary}
All harmonic maps of finite uniton number from a surface to $\O(n)$ are given locally as rational functions of a finite number of meromorphic functions $\mu_i$ and their derivatives of order less than or equal to $n-3$, together with the complex conjugates of those.
\qed \end{corollary}

\begin{remark}
(i) The algorithm in the proof finds the entries of $A$ in terms of \emph{generalized} derivatives of the $\mu_i$. However, these may be written in terms of \emph{ordinary} derivatives with respect to a local coordinate (and vice versa).
 
(ii) Although, in order to prove local surjectivity we have had to use integration,
\emph{the formulae we obtain for $A = h(\mu)$ are algebraic involving no integration, are globally defined on $\M(M)^p$  and are independent of local coordinates}.   
Further, there is a stratification of $\M(M)^p$ with top stratum $\M(M)^p_{\wt{ND}}$, with other strata determined by the list of degenerate steps in the algorithm, with different algebraic formulae on each stratum, see the examples in the next section.
 
(iii)  We could equally well give an algorithm with induction step which finds the new last \emph{column}, and then complete the matrix by finding the new first row and top-right element by algebra. That such a process would be equivalent to our method follows from Lemma \ref{lem:algebra}.

(iv) Our algorithm does not modify the parametrization of $\wt{A}$.  In the special case $n=6$ we get nicer parametrizations if we do that, see \S \ref{subsec:n=6}(c) and (e).
\end{remark}
 
\section{Classifications} \label{sec:exs}
 We use our algorithm to find all extended solutions
$\Phi:M \to \Omega_r\U(n)^{\rr}$ of canonical type, and so all harmonic maps $\varphi:M \to \O(n)$ of finite uniton number, 
in two cases: (i) $S^1$-invariant solutions of type $(1,1,\ldots,1)$; (ii) $n \leq 6$.
We shall interpret some of the resulting harmonic maps using terminology to be found in older papers, e.g. \cite{bahy-wood-G2}.
In \S \ref{subsec:totally-isotropic}, we will discuss how our constructions relate to totally isotropic holomorphic maps.    Recall from Remark \ref{rem:type111}(ii) and \cite[\S 6.3]{unitons} that
(i) if $n$ is odd, then the uniton number $r$ of $\varphi$ is even and $r \leq n-1$; 
(ii) if $n$ is even, then $r \leq n-2$.    

\subsection{$S^1$-invariant solutions of type $(1,1,\ldots,1)$}
\label{subsec:11...1}

Let $\xi_0$ denote the canonical element of type
$(1,1,\ldots,1)$, i.e., $\xi_0 = \ii\diag(n-1,n-2,\ldots,1,0)$; by Remark \ref{rem:type111}(ii), $n$ is odd. 
In this case, the algorithm of the last section becomes very simple and we can give a clearer statement.
When $n=1$, the only solution is $A=I$. Otherwise, the extended solution equation \eqref{GrM1} reads
\be{GrM1k-1}
a_{ik}' = \rho_{k-1,k}'\,a_{i,k-1} 
\quad \text{where} \quad \rho_{k-1,k} = a_{k-1,k}
\quad (i=1,\ldots,n, \ k=i+1,\ldots, n-1).
\ee

Note that the $\rho_{k-1,k} = a_{k-1,k}$ are the entries
of $A$ `in the $\g_1$-position', i.e., on the superdiagonal; we shall say that $A$ and the corresponding extended solution $\Phi=[A\ga_{\xi_0}]$ are
\emph{non-degenerate} if the superdiagonal elements $a_{k-1,k}$ of $A$ are non-constant, equivalently
their derivatives are not identically zero.  The weaker condition of non-degeneracy of $\wt{A}$ is also important; the development below shows that it is equivalent to our algorithm being non-degenerate.  
In either case,
we get the following more precise version of Theorem \ref{th:gen}.

\begin{theorem}\label{th:S1-11}
Let $M$ be a Riemann surface, $n =2m+1$ \ $(m \geq 0)$.
Given an $m$-tuple of meromorphic functions $(\mu_1,\ldots,\mu_m)$ on $M$,  
calculate generalized derivatives $\mu_i^{(j)}$ and functions $\rho_1,\ldots,\rho_{2m}$
inductively for $i =1,2,\ldots, m$ as follows$:$ 
\be{mu-derivs}
\mu_i^{(j)} = (\mu_i^{(j-1)})'/\rho_{m+i-j}' \quad (j=0,1,\ldots 2i-2);
\quad \quad \rho_{m-i+1} = \mu_i^{(2i-2)}, \ \rho_{m+i} = -\mu_i^{(2i-2)}.
\ee
Let $\M(M)^m_{\ND}$ be the space of $m$-tuples of meromorphic functions $(\mu_1,\ldots,\mu_m)$
satisfying the \emph{non-degeneracy condition:}
\be{mu-non-deg}
\mu_i^{(2i-2)} \quad \text{is non-constant for all} \quad i=1,\ldots, m.
\ee
Then there is a bijective map $h_0=h_0(\xi_0):\M(M)^m_{\ND} \to (\Sol^{\rr}_{\xi_0})_0^{\ND}$   
to the space of non-degenerate $S^1$-invariant extended solutions\/ $\Phi:M \to \Omega_{n-1}\U(n)^{\rr}$ of type $(1,1,\ldots,1)$ given
by $\Phi = [A \ga_{\xi_0}]$ for $A:M \to (\A_{\xi_0}^{\rr})_0$
where $A = A(\mu) =(a_{ij})$ is the unique unitriangular matrix with
\be{a-ij}
a_{ij} = \mu_{m+1-i}^{(2m+1-i-j)} \quad \text{for $i <j$ and $i+j \leq 2m+1$}.
\ee

The inverse is given by $\mu_i = a_{m-i+1\!,\,m+i}$ \ $(i=1,\ldots, m)$.

Slightly more generally, these formulae define a one-to-one correspondence between the subset $\M(M)^m_{\wt{ND}}  = \{\mu \in \M(M)^m:  \text{ \eqref{mu-non-deg} holds for } i=1,\ldots, m-1\}$ 
and the set
$\{A \in (\Sol^{\rr}_{\xi_0})_0 : \wt{A} \text{ is non-degenerate}\}$. 

The extended solution $\Phi= [A \ga_{\xi_0}]$ is given explicitly by \eqref{fact} where
$\al_i$ is the span of the last $i$ columns; the corresponding harmonic map $\varphi=\Phi_{-1}$ is given by \eqref{phi-S1-invt}.
\end{theorem}

\begin{proof}
This is trivially true for $m=0$, i.e.\ $n=1$, where $A$ is the $1 \times 1$ identity matrix, and there are no parameters.

Assume that it is true for $m$ replaced by $m-1$ for some $m > 0$, thus all solutions $\wt A = (a_{ij})_{i,j=2,\ldots, n-1}:M \to \A_{\wt\xi_0}^{\rr}$ to the extended solution equation
\eqref{GrM1k-1} are parametrized by an $(m-1)$-tuple
$(\mu_1,\ldots,\mu_{m-1})$ in the fashion described by the theorem.  
Following our algorithm, we add a border to give a square matrix $A$ of size $n$.  As usual, it suffices to find the new first row $(a_{11},\ldots, a_{1,n-1})$ by solving \eqref{GrM1k-1} for $i=1$. 
Of course,  $a_{11} = 1$, and the next entry $a_{12}$ satisfies no equation so we initially parametrize it by $\nu_0 = a_{12}$.
If $n=3$, there are no equations to satisfy and we complete the matrix by algebra, i.e.\ by using \eqref{cx-orthog}, see Example \ref{ex:357} below.

Otherwise, the first equation to satisfy in \eqref{GrM1k-1} is $a_{13}' = \rho_{23}'a_{12} = \rho_{23}'\nu_0$.  To integrate this, we replace $\nu_0$ by a new parameter $\nu_1 = a_{13}$ and set $\nu_0 = (\nu_1)^{(1)} :=(\nu_1)'/\rho_{23}'$. 
Substituting for $\nu_0$ in the expression $a_{12} = \nu_0$ gives $a_{12} = (\nu_1)^{(1)}$.
Inductively, to solve the $K$th equation $a_{1,K+2}' = \rho_{K+1,K+2}'a_{1,K+1}$, replace $\nu_{K\!\!-1}$ by a
new parameter  $\nu_K = a_{1,K+2}$ and set
$\nu_{K\!\!-1} = (\nu_K)^{(1)} := (\nu_K)'/\rho_{K+1,K+2}'$,
giving $(a_{12},a_{13}, \ldots, a_{1,K+2}) = 
\bigl((\nu_K)^{(K)},(\nu_K)^{(K\!\!-1)}, \ldots, \nu_K \bigr)$.

We end up with a final parameter $\nu = \nu_{n-3}$ such that the new first row is
$$
(a_{12},a_{13}, \ldots, a_{1,n-2},a_{1,n-1}) = 
(\nu^{(n-3)},\nu^{(n-4)},\ldots, \nu^{(1)}, \nu)
$$ 
where the generalized derivatives are given inductively
by $\nu^{(j)} = (\nu^{(j-1)})'/\rho_{n-2-j,n-1\!-j}'$ \ 
$(j=1,2, \ldots, n-3)$.
This agrees with \eqref{mu-derivs} as $\rho_i = a_{i,i+1} = \rho_{i,i+1}$ by \eqref{a-ij} for $i \leq m$ (and for $i >m$ by algebra).
The new last column and top-right entry $a_{1n}$ can now be found by algebra, i.e., by using $(c_i,c_n) = 0$ for $i=2,\ldots,n$.  Appending $\mu_m = \nu$ to the existing parameters $\mu_1,\ldots, \mu_{m-1}$ gives the desired parametrization $h_0$.
\end{proof} 

\begin{example} \label{ex:357}
For $n=7$, the theorem gives the following parametrization
by triples $(\mu_1,\mu_2,\mu_3)$ of meromorphic functions in
$\M(M)^3_{\wt{ND}} = \{(\mu_1,\mu_2,\mu_3) \in \M(M)^3: \mu_i^{(2i-2)} \text{ non-constant for } i=1,2 \}$ 
of all $A:M \to \A^{\rr}_{\xi_0}$ 
satisfying the extended solution equation with $\wt{A}$ non-degenerate.
This restricts to a parametrization by
$\M(M)^3_{\ND} = \{(\mu_1,\mu_2,\mu_3) \in \M(M)^3: \mu_i^{(2i-2)} \text{ non-constant for } i=1,2,3 \}$ of all non-degenerate $S^1$-invariant extended solutions of the maximum possible uniton number $6$.
The remaining entries $a_{ij}$ can be calculated by algebra, as below.

$$
A = \begin{pmatrix}
1 & \mu_3^{(4)} & \mu_3^{(3)} & \mu_3^{(2)} & \mu_3^{(1)} & \mu_3 & a_{17} \\
0 & 1 & \mu_2^{(2)} & \mu_2^{(1)} & \mu_2 & a_{26} & a_{27} \\
0 & 0 & 1         & \mu_1  & a_{35}    &a_{36} & a_{37}\\
0 & 0 & 0         & 1       & -\mu_1  & a_{46} & a_{47} \\
0 & 0 & 0         & 0       & 1       &  -\mu_2^{(2)} & a_{57} \\
0 & 0 & 0         & 0       & 0          & 1 & -\mu_3^{(4)} \\
0 & 0 & 0         & 0       & 0          & 0  & 1
\end{pmatrix}.
$$

To see this, we follow the induction starting with the middle $1 \times 1$ matrix which must be the identity matrix.  The middle $3 \times 3$ matrix is parametrized by
$\mu_1 (=a_{34})$, and its last column and top-right entry $a_{45}, a_{35}$ can be calculated from $(c_5, c_i)=0$ for $i=4,5$, thus $a_{45}=-\mu_1$ and
$a_{35} = -\half\mu_1^{\;2}$, cf. \S \ref{subsec:n=3}.
Then the middle $5 \times 5$ matrix, is parametrized by
$(\mu_1,\mu_2)$ and its last column and top-right entry $a_{56},\ldots, a_{26}$ can be calculated from $(c_6, c_i)=0$ for $i=3,4,5,6$, in particular $a_{56} = -\mu_2^{(2)}$, for the rest, see
 \S \ref{subsec:n=5}(c).  Finally, the $7 \times 7$ matrix $A$ is parametrized by $(\mu_1,\mu_2,\mu_3)$ and its last column and top-right entry $a_{67},\ldots, a_{17}$ can be calculated from $(c_7, c_i)=0$ for $i=2\ldots,7$;  in particular $a_{67} = -\mu_3^{(4)}$; the other entries $a_{i7}$ are polynomial in the $\mu_i$ and their derivatives: we leave the reader to work these out.
 
In degenerate cases, different formulae are obtained.  For example, if $\mu_1$ is constant, then by premultiplying by a suitable matrix $E$ as in the algorithm, we can make it $0$ and we obtain the middle $5 \times 5$ matrix in the right-hand matrix below.
Then, if $\mu_2$ is constant, again we can make it zero and we obtain the left-hand $7 \times 7$ matrix; if $\mu_2$ is not constant, we obtain the right-hand matrix.
$$
A = \begin{pmatrix}
1 & \mu_3 & 0 & 0 & 0 & 0 & 0 \\
0 &     1 & 0 & 0 & 0 & 0 & 0 \\
0 &     0 & 1 & 0 & 0 & 0 & 0 \\
0 &     0 & 0 & 1 & 0 & 0 & 0 \\
0 &     0 & 0 & 0 & 1 & 0 & 0 \\
0 &     0 & 0 & 0 & 0 & 1 & -\mu_3 \\
0 &     0 & 0 & 0 & 0 & 0 & 1
\end{pmatrix},
\qquad  
A = \begin{pmatrix}
1 & \mu_3^{(1)} & \mu_3 & 0 & 0  & 0 & 0 \\
0 & 1 & \mu_2 & 0 & 0 &  0 & 0 \\
0 & 0 & 1         & 0  &  0 & 0 &  0 \\
0 & 0 & 0         & 1       & 0  & 0 & 0 \\
0 & 0 & 0         & 0       & 1       &  -\mu_2 & \mu_2\mu_3^{(1)}-\mu_3 \\
0 & 0 & 0         & 0       & 0          & 1 & -\mu_3^{(1)} \\
0 & 0 & 0         & 0       & 0          & 0  & 1
\end{pmatrix}.
$$
Here $\mu_2$ and $\mu_3$ are arbitrary meromorphic functions and, in the right-hand matrix, $\mu_3^{(1)} := \mu_3'/\mu_2'$.
\end{example}

\subsection{Totally isotropic maps and extended solutions}
\label{subsec:totally-isotropic}

We now see how the extended solutions constructed in the last section relate to other interesting maps.
Recall (\cite{wolfson}, see also \cite[Example 4.7]{unitons}) that a harmonic map $f:M \to G_*(\C^n)$
 generates a \emph{harmonic sequence} $G^{(i)}(f)$ $(i \in \Z)$ of
\emph{Gauss bundles} or \emph{transforms}, all harmonic maps. 
By the \emph{(complex) isotropy order} of a \emph{harmonic} map $f:M \to \CP^{n-1}$, we mean the maximum $r$ such that $f$ is perpendicular to $G^{(i)}(f)$ for $i=1,\ldots,r$,
equivalently \cite[Lemma 3.1]{burstall-wood}, the maximum $r$ such that $G^{(i)}(f)$ is perpendicular to $G^{(j)}(f)$ for all $i, j \in \Z$ with $0 < |i - j| \leq r$. 

On the other hand, by the \emph{real isotropy order} of a full \emph{holomorphic} map
$f=[F]:M \to \CP^{n-1}$ we mean the maximum integer $t \geq -1$ such that
\be{isotropy}
(F^{(i)}, F^{(j)}) = 0  \quad \text{for all} \quad i,j \geq 0 \quad
\text{with} \quad i+j \leq t.
\ee
Here $F:U \to \C^n$ denotes a local holomorphic representative of\/ $f$ and $F^{(i)}$ denotes the $i$th derivative with respect to a local complex coordinate: 
the definition is independent of choice of\/ $F$ and of local coordinate.  Differentiation shows that, if
$(F^{(s)}, F^{(s)}) = 0$ for some $s$, then also
$(F^{(s+1)}, F^{(s)}) = 0$.  It follows that $t$ is odd, i.e.\ $t=2s+1$ for some $s \geq -1$; note that
$(F,F) = 0 \iff s \geq 0$.   The largest possible value of\/ $s$ is $[(n-3)/2]$:
in that case fullness implies that $n$ is odd and $t=n-2$, and we say that $f$ is \emph{totally isotropic} \cite{eells-wood}.
Note that the real isotropy order $t$ is not the same as the complex isotropy order: indeed, the latter is infinite for a holomorphic map. However, if $f$ is a holomorphic map of real isotropy order $t \geq 0$, the map $f \oplus \ov{f}:M \to G_2(\R^n)$ is a harmonic map called a \emph{real mixed pair};
by \cite[Lemma 2.14]{bahy-wood-G2} this has complex isotropy order $t$.

In \cite{calabi-quelques,calabi-JDG}, E.~Calabi showed how that all harmonic maps into $\RP^{2m}$ or $S^{2m}$ can be obtained from totally isotropic holomorphic maps, giving the
 bijections between (ii), (iii) and (iv) below; in particular, the bijection from (ii) to (iii) is given by $f \mapsto G^{(m)}(f)$.  We now explain how these relate to polynomial extended solutions of harmonic maps into $\O(2m+1)$ of type $(1,1,\ldots,1)$, and so of the maximum possible uniton number $2m$. The corresponding canonical element is $\xi_0 = \ii\diag(2m,2m-1,\ldots, 1,0)$.

\begin{theorem} \label{th:11-appl}
Let $M$ be a Riemann surface and $n =2m+1 \geq 3$ an odd integer. The following sets are in one-to-one correspondence$:$
\begin{itemize}
\item[(i)] non-degenerate $S^1$-invariant extended solutions
$\Phi:M \to \Omega_{n-1}\U(n)^{\rr}$ of type $(1,1,\ldots,1);$

\item[(i)$'$] non-degenerate solutions $A:M \to (\A_{\xi_0}^{\rr})_0$ to the extended solution equation \eqref{GrM}$;$

\item[(ii)] full totally isotropic holomorphic maps $f:M \to \CP^{n-1};$

\item[(iii)]  full harmonic maps $\varphi:M \to \RP^{n-1};$

\item[(iv)] antipodal pairs $\pm\wt\varphi:M \to S^{n-1}$ of full
harmonic maps.
\end{itemize}
In particular, we obtain an explicit algebraic parametrization of sets (i)--(iv) by 
$m$-tuples $(\mu_1,\ldots,\mu_m)$ of meromorphic functions satisfying the non-degeneracy condition \eqref{mu-non-deg}.
\end{theorem}

\begin{proof}
By Proposition \ref{prop:BuGu-On}, the map $\Phi = [A\ga_{\xi_0}]$ defines a bijection between
(i) and (i)$'$. 
Given $A$ in (i)$'$, its last column gives a full totally isotropic holomorphic map $f$; indeed, each associated curve $f_{(i)}$ is the span of the last $i+1$ columns of\/ $A$, so that $f$ is full and \eqref{isotropy} holds for $t = n-2$; thus $f$ is in set (ii).

Conversely, given $f$ in (ii), we can write $f = [F]$ where $F = (F_0,F_1,\ldots, F_{n-1})$ is meromorphic with $F_0 = 1$; define the last column of\/ $A$ by $c_n = F^{\TT}$, i.e., 
$a_{in} = F_{n-i}$ \ $(i=1,\ldots, n)$.  Then, for
$j=n-1,n-2, \ldots, 1$, define the $j$th column of\/ $A$ by
$c_j = c_{j+1}'/a_{j,j+1}'$; by fullness, no denominator is identically zero; this gives $A$ in (i)$'$.

The last statement follows by parametrizing set (i) as in Theorem \ref{th:S1-11}. 
\end{proof}

\subsection{Uniton number at most 2} \label{subsec:r=1,2}
In this case, we find all harmonic maps completely explicitly, as follows. In the sequel, all uniton factorizations will be the alternating factorization, see \S \ref{subsec:uniton-fact}.

\begin{proposition} \label{prop:r=0,1}
{\rm (i)} A harmonic map $\varphi:M \to \O(n)$ has uniton number $0$ if and only if it is constant; in particular, any harmonic map with $n \leq 2$ is of this type.
It has an associated extended solution $\Phi=[A\ga_{\xi}] = I$ of type $(n)$ given by $\xi = 0$ and $A = \ga_{\xi} = I$.  We shall refer to this as the \emph{trivial solution}.

{\rm (ii)} A harmonic map $\varphi:M \to \O(n)$ has uniton number $1$ if and only if $n=2m$
for some $m$, and up to left-multiplication by a constant matrix in $\O(2m)$, it is a holomorphic map into $\O(2m)/\U(m)$.   More precisely, $\varphi = \ii(\pi_V - \pi_V^{\perp}) = \ii(\pi_V - \pi_{\ov{V}})$ where $V$ is a maximally isotropic holomorphic subbundle of\/ $\CC^n;$ $\varphi$ has
 associated extended solution $\Phi = \pi_V + \la \pi_V^{\perp}$ with $\varphi = \ii\Phi_{-1}$.
\end{proposition}

\begin{proof}  (i) Evident, since it has a polynomial associated extended solution of degree
$0$, which must equal the identity matrix.  When $n \leq 2$ this is the only element of
$\O(n,\C)$.

(ii) By Proposition \ref{prop:BuGu-On}, $\varphi$ has an associated extended solution $\Phi=[A\ga_{\xi}]$
of canonical type with $r=1$.  The type must thus be $(m,m)$ for some $m$, so $n=2m$ and
$\xi = \ii\diag(1,\ldots,1,0,\ldots,0)$, which gives the canonical geodesic $\ga_{\xi} = \diag(\la,\ldots,\la,1,\ldots,1)$.
Now, any solution $A:M \to \O(2m,\C)$ to \eqref{GrM} with $r=1$
 is of the form\\[-4ex]
\begin{equation} \label{A-r1}
A = \begin{pmatrix} I & B \\ 0 & I \end{pmatrix}
\end{equation}
where $B:M \to \gl(m,\C)$ is meromorphic and has $B^{\TT} = -B$.
The resulting Grassmannian model is $W = V + \la\HH_+$ where $V$ is the span of the last $m$ columns of $A$, a maximally isotropic subbundle of $\CC^n$, equivalently a holomorphic map from $M$ to $\O(2m)/\U(m)$ (see \S \ref{subsec:On-1}); this $W$ corresponds to the stated extended solution.
\end{proof}
See \S \ref{subsec:n=4}(a)  and \S \ref{subsec:n=6}(b) for examples of this type.
We next discuss extended solutions of harmonic maps of uniton number $2$. 
Such a harmonic map has an associated polynomial extended solution of canonical type and of degree $2$, so it suffices to discuss those.

\begin{proposition} \label{prop:r=2}
{\rm (i)} Any extended solution $\Phi:M \to \Omega_2\U(n)^{\rr}$ of canonical type has a uniton factorization of the form
\be{r=2-gen}
\Phi = (\pi_X + \la\pi_X^{\perp})(\pi_V + \la\pi_V^{\perp})
\ee 
where $X$ and $V$ are holomorphic subbundles of\/ $\CC^n$ with $X^{\perp}$ and $V$ isotropic and $\pi_V X = V$.

This is $S^1$-invariant if and only if\/ $X$ is the polar\/
$V^{\circ} = \ov{V}^{\perp}$ of\/ $V$, in which case $X$, $V$ and $\ov{V}$ all commute and \eqref{r=2-gen} reads 
\be{r=2-S1-invt}
\Phi = \la(\pi_{\ov{V}} + \la^{-1}\pi_{\ov{V}}^{\perp})(\pi_V + \la\pi_V^{\perp})
	= \la(\pi_V + \la\pi_V^{\perp})(\pi_{\ov{V}} + \la^{-1}\pi_{\ov{V}}^{\perp}).
\ee
The corresponding harmonic map is then $\varphi = \Phi_{-1} = V \oplus \ov{V}:M \to G_{2s}(\C^n)$ (where $s=\rank V$), which is
a (higher dimensional) \emph{real mixed pair} \cite{bahy-wood-G2}, and has (minimal) uniton number $2$ unless $V$ is constant.

{\rm (ii)} All extended solutions $\Phi:M \to \Omega_2\U(n)^{\rr}$ of type $(1,t_1,1)$ are
$S^1$-invariant, and so are of the form \eqref{r=2-S1-invt} with
$\rank V=1$; the corresponding harmonic maps $\varphi = V \oplus \ov{V}$ are real mixed pairs.
\end{proposition}

\begin{proof}
(i) Write $\Phi = [A\ga_{\xi}]$; note that the type must be $(t_0,t_1,t_0)$ for some
$t_0, t_1$ with $2t_0 + t_1 = n$.
For each $j$, write the $j$th column of $A$ as $c_j = c_j^0 + \la c_j^1$, note $c_j^1=0$ for all $j \leq t_0+t_1$.
Set $X = \spa\{c_j^0 : t_0 < j \leq n\}$ and
$\wt{V} = \spa\{c_j = c_j^0+ \la c_j^1 : t_0+t_1 < j \leq n\}$.  Then the Grassmannian model $W = A\ga_{\xi}\H_+$ is
$W = \wt{V} + \la X + \la^2 \H_+$ so, from \eqref{alt-fact}, \eqref{alpha} and \eqref{fact}, the alternating uniton factorization
is given by \eqref{r=2-gen} where
$V = \spa\{c_j^0 + \pi_X^{\perp}c_j^1 : t_0+t_1 < j \leq n\}$.  

This is $S^1$-invariant if and only if $c_j^1 = 0$ for all $t_0+t_1  <j \leq n$, equivalently $X$ is the polar of $V$.
Then $(V,X)=(V,V^{\circ})$ is a \emph{$\pa'$-pair} in the sense of \cite{erdem-wood}.
Thus the Grassmannian model $W = A\ga_{\xi}\H_+$ is $W = V + \la V^{\circ} + \la^2 \HH_+$, giving extended solution \eqref{r=2-S1-invt}.

(ii) By Lemma \ref{lem:top-right-deg}, the maximum degree of any term of $A$ is $0$, giving an $S^1$-invariant extended solution.
\end{proof}

\subsection{All extended solutions for $n$ at most $6$} \label{subsec:n-leq-6}
We will now find all extended solutions of canonical type for $n \leq 6$.  To do this
we find all solutions $A:M \to \A_{\xi}^{\R}$ to \eqref{GrM} by our algorithm; we can then compute the corresponding extended solutions $\Phi=[A\ga_{\xi}]$ using the formulae in
\S \ref{subsec:uniton-fact}, or \S \ref{subsec:cx-Grass} in the $S^1$-invariant case.    By modifying our algorithm, and so the mappings $h$ and $h_0$ in some cases, we obtain the following improvement of Theorem \ref{th:gen}
where `locally surjective' is replaced by `surjective', or even, `bijective'.
 
\begin{theorem} 
Let $M$ be a Riemann surface and let $n \leq 6$.  Let\/ $\xi=\ii\diag(\xi_1,\ldots,\xi_n)$ be a canonical element of\/ $\Omega_r\U(n)^{\rr}$ for some $r \in \N;$
set $p = p(\xi) = \sum_{k=1}^r \dim \g_k^{\rr}(\xi)$ and
$p_1 = p_1(\xi) = \dim \g_1^{\rr}(\xi)$.  There are maps
$h=h_{\xi}:\M(M)^p \to \Sol_{\xi}^{\rr}$ and $h_0 = (h_0)_{\xi}:\M(M)^{p_1} \to (\Sol_{\xi}^{\rr})_0$ such that
\begin{enumerate}
\item[(i)] $h$ and $h_0$ are \emph{surjective} up to replacing $A$ by $EA$ for some constant matrix
$E \in \O(n,\C);$

\item[(ii)] $h$ restricts to a \emph{bijective} map $h:\M(M)^p_{\wt{\ND}} \to (\Sol_{\xi}^{\rr})_{\wt{\ND}}$, $\mu \mapsto A(\mu)$ to an open dense subset of\/ $\Sol_{\xi}^{\rr}$; this map
is algebraic in the sense that each entry of\/ $A(\mu)$ is polynomial
in $\la$ with coefficients rational functions of the $\mu_i$ and their derivatives of order less than or equal to $n-3$.
\end{enumerate}  
\end{theorem}

We shall show this for each dimension in turn, concentrating on non-degenerate cases; the reader can easily calculate degenerate cases as in Example \ref{ex:357}.  Dimensions $n=1$ and $2$ are trivial, see Proposition
\ref{prop:r=0,1}, so we start with $n=3$.

\subsection{Dimension n=3} \label{subsec:n=3}
All solutions are obtained from the unique $n=1$ case $\wt A=(1)$ by adding a border.  This gives 
one non-trivial type, $(1,1,1)$, i.e., $\xi = \ii\diag(2,1,0)$ giving the closed geodesic
$\ga_{\xi} = \diag(\la^2, \la, 1)$. Any solution $A:M \to \O(3,\C)$ to \eqref{GrM} is
obtained from the identity matrix in $\O(1,\C)$ by choosing an arbitrary meromorphic function 
$a_{12} = -g$, say; in fact, this is the lowest-dimensional case of Theorem
\ref{th:S1-11} as in Example \ref{ex:357} (with $\mu_1 = -g$).  Filling in the last column by algebra, i.e., using $(c_i,c_3)=0$ for $i=2,3$
(see \S \ref{subsec:add-border}) gives a complex extended solution $\Psi = A \ga_{\xi}$ where
\be{n=3}
 A = \begin{pmatrix} 1 & -g & -\half g^2 \\
					0 & 1 & g \\
					0 & 0 & 1 \end{pmatrix}.
\ee
Let $\Phi = [A \ga_{\xi}]$ and $\varphi = \Phi_{-1}$ be the corresponding extended solution and harmonic map.
 As in Proposition \ref{prop:r=2}(ii), 
$\varphi: M \to G_2(\R^3)$ is the real mixed pair
given by $\varphi^{\perp} = V \oplus \ov{V}$.
More explicitly, let $Q_{n-2}$ denote the complex quadric
$\{[z_0,\ldots,z_{n-1}] \in \CP^{n-1}: \sum_{i=0}^{n-1}z_i z_{n-i-1} = 0\}$; then,  with $\iota$ denoting the Cartan embedding, $\varphi$ is the composition:
 $$
\begin{gathered}
\xymatrixrowsep{-0.3pc}\xymatrixcolsep{1pc}\xymatrix{
M \ar[r]^g & \Cinfty \ar[r]^{\equiv} & \CP^1 \ar[r]^{\equiv}& Q_1
\ar[r]^{\equiv} & S^2 \ar[r]^{2:1} & G_2(\R^3) \ar[r]^{\iota} & \O(3)  \\
z \ar[r]& g = g(z)\ar[r] & [1,g] \ar[r]& h=[1,g,-\half g^2]\ar[r] & (h\oplus\ov{h})^{\perp}
	\ar[r] & h \oplus \ov{h}\ar[r] & \pi_{h\oplus\ov{h}} - \pi_{h\oplus\ov{h}}^{\perp} 
}
\end{gathered}
$$
Here and in the rest of the paper, $\equiv$ denotes a standard identification.
The real line $(h\oplus\ov{h})^{\perp}$ is given a canonical orientation so that it gives a point of $S^2$; the composition $\Cinfty \to \CP^1 \to Q_1 \to S^2$ of the maps above is stereographic projection.

Note that, if $g$ is constant, then $\varphi$ is constant and has (minimal) uniton number $0$, otherwise it has uniton number $2$.

\subsection{Dimension n=4} \label{subsec:n=4}
There are two non-trivial types, as follows. 

\medskip

(a) {\bf Type (2,2)}. Here $r=1$ and $\xi = \ii\diag(1,1,0,0)$,
and, as in Proposition \ref{prop:r=0,1}(ii),
$$
A = \begin{pmatrix} 1 & 0 &-g & 0 \\
					0 & 1 & 0 & g \\
					0 & 0 & 1 & 0 \\
					0 & 0 & 0 & 1 \end{pmatrix}.
$$
for some arbitrary meromorphic function $g$ on $M$.
Then $W = V \oplus \la \HH_+$, where $V$ is the maximally 
isotropic subbundle of $\CC^4$ spanned by the last two columns $c_3$ and $c_4$
of $A$ and the extended solution $\Phi=[A\ga_{\xi}]$ is $\Phi=\pi_V+\la\pi_V^{\perp}$. The corresponding harmonic map $\Phi_{-1}$ is the holomorphic map $V:M \to \O(4)/\U(2)$.
More explicitly, it is the composition:
$$
\begin{gathered}
\xymatrixrowsep{-0.3pc}\xymatrixcolsep{2pc}\xymatrix{
M \ar[r]^g & \Cinfty \ar[r]^{\equiv} & \CP^1 \ar[r]^{\equiv} & \SO(4)/\U(2) \ar[r]^{\text{inclusion}} & \O(4)/\U(2) \ar[r]^{\iota} & \O(4) \\
z \ar[r] & g = g(z) \ar[r] & [1,g] \ar[r] & V = \spa{(c_3,c_4)} \ar[r] & V \ar[r]
& \ii(\pi_V - \pi_{\ov V})	 
}
\end{gathered}
$$

\medskip

(b) {\bf Type (1,2,1)}. Here $r=2$, the maximum possible for $n=4$, and $\xi = \ii\diag(2,1,1,0)$.
We obtain the solution by adding a border to the unique solution $\wt{A}=I$ of type (2).  Then we have two new entries $a_{12}$, $a_{13}$ in the $\g_1$-position (i.e., on the block superdiagonal), we set
 $a_{12} = -g_1$, $a_{13} = -g_2$ where $g_1,g_2$ are arbitrary meromorphic functions.  Filling in the last column by algebra
(see \S \ref{subsec:add-border}) gives
$$
A = \begin{pmatrix} 1 & -g_1 & -g_2 &  -g_1 g_2 \\
					0 &   1  &   0  &   g_2 \\
					0 &   0  &   1  &   g_1 \\
					0 &   0  &   0  &    1  \end{pmatrix}.
$$

Let $h$ denote the span of the last column $c_4$, 
thus $h = [1,g_1,g_2,-g_1g_2]^{\TT}$ where $\hphantom{.}^{\TT}$ denotes the second transpose as in \S \ref{subsec:On}; by \eqref{cx-orthog} the polar
$h^{\circ} = \ov{h}^{\perp}$ of $h$ is the span of the last three columns.  The above $A$ gives
$W = h \oplus \la h^{\circ} \oplus \la^2\HH_+$,
and as in Proposition \ref{prop:r=2}, the corresponding extended solution is
$$
\Phi = (\pi_h + \la \pi_h^{\perp}) 
					(\pi_{h^{\circ}} + \la \pi_{h^{\circ}}^{\perp})
=
\la (\pi_h + \la \pi_h^{\perp})(\pi_{\ov{h}} + \la^{-1} \pi_{\ov{h}}^{\perp}).
$$
This is an extended solution of the real mixed pair
$h \oplus \ov{h}: M \to G_2(\R^4)$ (or its orthogonal complement).
More explicitly, it is the composition:
$$
\begin{gathered}
\xymatrixrowsep{-0.3pc}\xymatrixcolsep{2pc}\xymatrix{
M \ar[r]^{\!\!\!\!\!\!\!\!\!\!\!\!\!\!\!\!\!\!\!\!\!\!(g_1,g_2)} & \Cinfty\times\Cinfty \ar[r]^{\equiv} & \CP^1\times\CP^1 \ar[r]^{\:\:\:\:\equiv}& Q_2
	\ar[r]^{\!\!\!\!\!\!2:1} & G_2(\R^4) \ar[r]^{\iota} & \O(4) \\
z \ar[r] & (g_1,g_2) = (g_1(z),g_2(z)) \ar[r] & ([1,g_1],[1,g_2]\ar[r] & h \ar[r] &  h\oplus\ov{h}
	\ar[r] & \pi_{h\oplus\ov{h}} - \pi_{h\oplus\ov{h}}^{\perp}
}
\end{gathered}
$$

Up to now, there have been no equations to satisfy and no terms in $\la$; this shows the following, which is a consequence of \cite[Proposition 6.20]{unitons}.

\begin{proposition} \label{prop:n-leq-4} 
When $n \leq 4$,
\begin{enumerate}
\item[(i)] every extended solution of canonical type  $M \to \Omega\O(n)$ is $S^1$-invariant$;$
\item[(ii)] in particular, every extended solution of canonical type  $M \to \Omega\O(n)$ satisfies the symmetry condition
\eqref{Phi-Grass}, and so the corresponding harmonic map $\varphi = \Phi_{-1}$ maps into a real Grassmannian or into $\O(2m)/\U(m)$ with $n = 2m$.
\end{enumerate}
\end{proposition}

That neither statement is true for $n=5$ is shown by the examples in \S \ref{subsec:n=5}(a) and (c) below.

Note that the proposition together with the last statement of Proposition \ref{prop:BuGu-On} shows that \emph{every harmonic map of finite uniton number from a surface to $\O(n)$ with $n \leq 4$ has an associated extended solution which is $S^1$-invariant}.

\subsection{Dimension n=5} \label{subsec:n=5}
All solutions are obtained from one of the two $n=3$ cases of \S \ref{subsec:n=3}, i.e., type (3) or type (1,1,1), by adding a border.  This gives
three non-trivial types, as follows.

\medskip

(a) {\bf Type (2,1,2)}. Here $r=2$ and $\xi = \ii\diag(2,2,1,0,0)$.
We apply the algorithm in the proof of Theorem \ref{th:gen} to obtain this case from
the $(1,1,1)$ case \eqref{n=3}; we shall give the details in the non-degenerate case, i.e., when $g$ is non-constant.
We have one new entry $a_{13}$ in the $\g_1$-position; we initially set this equal to 
an arbitrary meromorphic function $\nu_1$.  Write
$a_{14} = a_{14}^0 + \la a_{14}^1$.  Then  $a_{14}^1$ is arbitrary, say $\sigma$, and $a_{14}^0$ satisfies
$(a_{14}^0)' = -g' a_{13}$ $\mod \la$. 
According to the algorithm, to integrate this, we replace our initial choice $\nu_1$ of parameter by a new parameter $\nu = a_{14}^0$ so that $\nu_1 = \nu^{(1)}$, where generalized derivatives
$\nu^{(d)}$ are taken with respect to $g$.
  As no further integrations are necessary, $\nu$ is our final parameter. Then, filling in the last column by algebra, i.e., using $(c_i,c_5)=0$ for $i=3,4,5$
(see \S \ref{subsec:add-border}), we obtain
$$
A = \begin{pmatrix}
1 & 0 & -\nu^{(1)} & \nu + \la\sigma & -\half(\nu^{(1)})^2 \\
0 & 1 & g            & -\half g^2   & -\nu + g \nu^{(1)} - \la\sigma \\
0 & 0 & 1            &  -g        & \nu^{(1)} \\
0 & 0 & 0            &   1        & 0 \\
0 & 0 & 0            &   0        & 1 \end{pmatrix}.
$$

By Remark \ref{rem:A}(v), this gives an $S^1$-invariant extended solution $\Phi=[A\ga_{\xi}]$ if and only if $\sigma \equiv 0$, i.e., $\sigma$ is identically zero;
in which case it has corresponding harmonic map $\varphi = \al_1 \oplus \ov{\al_1}$ where $\al_1$ is the span of the last two columns.
Define $h:M \to \CP^4$ as the span of $c_5 + \nu^{(2)} c_4$.  When $\nu^{(3)} \not\equiv 0$, the last two columns $c_4,c_5$ are spanned by $h$ and its derivative, thus 
$\varphi = h_{(1)} \oplus \ov{h_{(1)}}: M \to G_4(\R^5)$. Its orthogonal complement is the harmonic map $\varphi^{\perp}:M \to \RP^4$ given by the middle vertex of the following harmonic sequence ---
 by being careful with orientations $\varphi^{\perp}$ actually defines a map into $S^4$.
$$
h \to G^{(1)}(h) \to \varphi^{\perp} \to \ov{G^{(1)}(h)} \to \ov{h}.
$$
If $\sigma$ is not identically zero, then the harmonic map $\Phi_{-1}$ does not lie in a Grassmannian.

\begin{remark} This example is equivalent to that of \cite[Example 6.21]{unitons}.
The reality conditions (i)--(iii) of that example, which were hard to solve using the methods of \cite{unitons}, are automatically satisfied by our method.
\end{remark}

\medskip

(b) {\bf Type (1,3,1)}, so $r=2$. This is obtained from $n=3$, type (3), i.e., the identity matrix, by adding a border giving
$$
A = \begin{pmatrix}
1 & \nu_1 & \nu_2    & \nu_3&   -\nu_1\nu_3 - \half \nu_2^{\;2} \\
0 & 1 & 0            &   0        & -\nu_3   \\
0 & 0 & 1            &   0        & -\nu_2 \\
0 & 0 & 0            &   1        & -\nu_1\\
0 & 0 & 0            &   0        & 1 \end{pmatrix}.
$$
The resulting extended solution and harmonic map are described by Proposition \ref{prop:r=2}(ii).

\medskip 

(c) {\bf Type (1,1,1,1,1)}. Here $r=4$ and $\xi = \ii\diag(-2,-1,0,1,2)$.
As in the $(2,1,2)$ case above, we apply the algorithm in the proof of Theorem \ref{th:gen} to obtain this case from the $(1,1,1)$ case \eqref{n=3}.
As in Theorem \ref{th:S1-11}, this shows that any $S^1$-invariant extended solution
with middle $3 \times 3$ matrix $\wt{A}$ non-degenerate, i.e., $g$ non-constant,
has a complex extended solution $\Psi = A\ga_{\xi}$ where
\be{11111}
A = \begin{pmatrix}
1 & - \nu_1^{(2)} & \nu_1^{(1)} & \nu_1  
							& \nu_1 \nu_1^{(2)}-\half {\bigl(\nu_1^{(1)}\bigr)}^{\!2} \\
0 & 1         & -g      & -\half g^2 & -\nu_1 + g\nu_1^{(1)} - \half g^2 \nu_1^{(2)} \\
0 & 0         & 1       & g          & -\nu_1^{(1)} + g \nu_1^{(2)} \\
0 & 0         & 0       & 1          &  \nu_1^{(2)}  \\
0 & 0         & 0       & 0          & 1 \end{pmatrix}
\ee
for arbitrary meromorphic functions $g$ and $\nu_1$ with $g$ non-constant, and generalized derivatives are taken with respect to $g$. When $\wt{A}$ is degenerate, i.e.\ $g$ is constant, we obtain a simpler formula, see Example \ref{ex:357}.

Note that $A$ itself is non-degenerate if and only if both $g$ and $\nu_1^{(2)}$ are non-constant; equivalently, the last column spans a full holomorphic map $h:M \to \CP^n$.
Then $\varphi=\Phi_{-1}$ is the harmonic map
$\varphi = h \oplus G^{(2)}(h) \oplus G^{(4)}(h)$;
as in Theorem \ref{th:11-appl}, $h$ totally isotropic, i.e., $G^{(4)}(h) = \ov{h}$, so that
$\varphi$ is a harmonic map into the \emph{real} Grassmannian $G_3(\R^5)$.   Also, $G^{(2)}(h)$ defines a harmonic map into $\RP^4$ and into its double cover $S^4$.
Finally note that the middle three components of $h$ give a `null curve' in $\C^3$, see \S \ref{subsec:null-R3}.  

We now look for the general solution $A = A_0 + \la A_1 + \la^2 A_2$ with $A_0$ non-degenerate.  Of course,
$A_0$ is given by \eqref{11111}, but there are now two more initial parameters $\nu_2^1$
and $\nu_3^1$ with  $a_{13}^1 = \nu^1_2$ and $a_{14}^2 = \nu_3^1$.
As in the proof of Theorem \ref{th:gen}, we have to satisfy the equation
$$
(a^1_{14})' = \rho_{34}'a^1_{13} + \rho_{24}'a^0_{12} = g'\nu^1_2\,.
$$
Following our algorithm, we replace $\nu_2^1$ by $\nu_2^2 = a^1_{14}$ where
$(\nu^2_2)^{(1)} := (\nu^2_2)'/g' = \nu^1_2$ so that
$a^1_{13} = (\nu^2_2)^{(1)}$. Writing $\nu_2 =\nu^2_2$ and $\nu_3= \nu_3^1$ for our final choice of parameters, this gives
$$
\la A_1 + \la^2 A_2 = \begin{pmatrix}
0 & 0 & \la \nu_2^{(1)} & \la \nu_2 + \la^2 \nu_3 & \la\zeta_1 + \la^2 \zeta_2 \\
0 & 0 & 0 & 0 & -\la(\nu_2 - g \nu_2^{(1)}) - \la^2 \nu_3 \\
0 & 0 & 0 & 0 & -\la \nu_2^{(1)} \\
0 & 0 & 0 & 0 & 0 \\
0 & 0 & 0 & 0 & 0 \end{pmatrix}.
$$
where $\zeta_1$ and $\zeta_2$ are determined by algebra.
To see what harmonic map this gives, writing $H_i$ for the last column of $A_i$ \ $(i=0,1,2)$ and $h=\spa{H_0}$, we have
$$
W = \spa(H_0+\la H_1 + \la^2 H_2) \oplus \la\{\spa(H_0 + \la H_1)\}_{(1)} + \la^2 h_{(2)} \oplus \la^3 h_{(3)} + \la^4 \HH_+\,.
$$
By Remark \ref{rem:A}(v), this satisfies the symmetry condition \eqref{Phi-Grass} (and so gives a harmonic map into a Grassmannian) if and only if the parameter $\nu_2 \equiv 0$,
equivalently $H_1 \equiv 0$. It gives an $S^1$-invariant solution if and only if $\nu_2 \equiv \nu_3 \equiv 0$,
equivalently $H_1 \equiv H_2 \equiv 0$.
In all cases, we can find the alternating factorization \eqref{fact} of $\Phi$ into unitons by using \eqref{alt-fact} and \eqref{alpha}.
 We work this out for the Grassmannian case $\nu_2\equiv 0$:  for simplicity we write $H=H_0$ and $K=H_2$; then the extended solution is
$$
\Phi = (\pi_{\al_1} + \la\pi_{\al_1}^{\perp}) 
(\pi_{\al_2}+\la\pi_{\al_2}^{\perp})(\pi_{\al_3}+\la\pi_{\al_3}^{\perp}) (\pi_{\al_4}+\la\pi_{\al_4}^{\perp})
$$
where 
$\al_1 = h_{(2)}$, $\al_2 = h_{(1)}$,
$\al_3 = \beta \oplus G^{(2)}(h) \oplus G^{(3)}(h)$
with $\beta = \spa\{H + \pi_{h_{(2)}}K,\,H' + \pi_{h_{(2)}}K'\}$, and
$\al_4 = \spa\{H + \pi_{h_{(2)}}K\}$.
Note that $\al_2$ (resp. $\al_4$) is isotropic and is the polar of
$\al_1$ (resp. $\al_3$).  Thus $\Phi$ is the product of two `real' quadratic factors:
\be{Phi-11111}
\Phi = (\pi_{\al_2} + \la\pi_{\al_2 \oplus \ov{\al_2}}^{\perp} + \la^2\pi_{\ov{\al_2}}) (\pi_{\al_4} + \la\pi_{\al_4 \oplus \ov{\al_4}}^{\perp} + \la^2\pi_{\ov{\al_4}}).
\ee
This gives the harmonic map $\varphi=\Phi_{-1}$ as a product of two maps into $\O(5)$: 
\be{varphi-11111}
\varphi = (\pi_{\al_2 \oplus \ov{\al_2}} - \pi_{\al_2 \oplus \ov{\al_2}}^{\perp})
(\pi_{\al_4 \oplus \ov{\al_4}} - \pi_{\al_4 \oplus \ov{\al_4}}^{\perp}),
\ee
which is the map $M \to G_3(\R^5)$ given by
$\varphi = \al_4 \oplus G^{(2)}(h) \oplus \ov{\al_4}$.

When $H_1 \equiv 0$ but $H_2 \not\equiv 0$, \eqref{varphi-11111} gives an example of a harmonic map into $G_2(\R^5)$ with non-$S^1$-invariant extended solution.
When $H_1 \not\equiv 0$, \eqref{Phi-11111} doesn't satisfy the symmetry condition \eqref{Phi-Grass}.  Thus, in this example,
the corresponding harmonic map $\Phi_{-1}$ does not, in general, have values in a Grassmannian.  By Proposition \ref{prop:n-leq-4}, this cannot happen in dimension $n \leq 4$.

\subsection{Dimension n=6} \label{subsec:n=6}
All solutions are obtained from one of the three $n=4$ cases in \S \ref{subsec:n=4} by adding a border.  This gives five non-trivial types, as follows.

\smallskip

(a) {\bf Type (1,4,1)}, so $r=2$. This is similar to $n=5$, type $(1,3,1)$ above. 

\medskip

(b) {\bf Type (3,3)}.  This has $r=1$ and is obtained from type $(2,2)$ by adding a border; there are two new parameters $\nu_1,\nu_2$ in the
$\g_1$-position, call these $-h$ and $-k$ giving the $S^1$-invariant solution depending on three arbitrary meromorphic functions:
$$
A = \begin{pmatrix} 1 & 0 & 0 & -h & -k & 0 \\
					0 & 1 & 0 & -g &  0 & k \\
					0 & 0 & 1 &  0  & g & h \\
					0 & 0 & 0 &  1 &  0 & 0 \\
					0 & 0 & 0 &  0 &  1 & 0 \\
					0 & 0 & 0 &  0 &  0 & 1 \end{pmatrix}.
$$
By Proposition \ref{prop:r=0,1}(ii), the corresponding harmonic map is the holomorphic map
$V: M \to \SO(6)/\U(3)$
given by the maximally isotropic subspace $V$ spanned by the last three columns of $A$.
Now the holomorphic map $\C^3 \to \SO(6)/\U(3)$ given by $(g,h,k) \mapsto V$  extends to a holomorphic diffeomorphism from $\CP^3$ to $\SO(6)/\U(3)$ given by $[\ell,g,h,k] \mapsto $ the span of the four vectors
$(0,0,\ell,0,-g,-h)$, $(0,\ell,0,g,0,-k)$, $(\ell,0,0,h,k,0)$, $(g,-h,k,0,0,0)$; 
whether $d$ is zero or non-zero, these vectors are linearly dependent and span a maximally isotropic subspace of dimension $3$, cf.\ \cite[Example 2.4]{baird-wood-higher} or
\cite[\S 3.1]{borisov-salamon-viaclovsky}, thus $V$ defines a holomorphic map into $\CP^3$. 

\medskip

(c) {\bf Type (2,2,2)}.  This has $r=2$ and is obtained from type (1,2,1) in \S \ref{subsec:n=4}
above by adding a border.
The entries in the first row in the $\g_1$-position  are $a_{13}$ and $a_{14}$, giving two new parameters, and the $\la$-term of $a_{15}$ gives a further parameter.
Carrying out our algorithm in the case that $g_1$ and $g_2$ are non-constant gives
$$
A = \begin{pmatrix} 1 & 0 & (\nu_1)^{(1)} & (\nu_2)^{(1)} & \nu_1+\nu_2+\la\nu_3  & a_{16} \\
					0 & 1 & -g_1   & -g_2 &  -g_1 g_2 & a_{15} \\
					0 & 0 &   1    &   0  &   g_2     & a_{14} \\
					0 & 0 &   0    &   1  &   g_1     & a_{13} \\
					0 & 0 &   0    &   0  &    1      & 0 \\
					0 & 0 &   0    &   0  &    0      & 1
\end{pmatrix}.
$$
Here $(\nu_1)^{(1)} = \nu_1'/g_2'$ and $(\nu_2)^{(1)} = \nu_2'/g_1'$, and our final new parameters are
$\nu_1$, $\nu_2$ and $\nu_3$, together with the existing parameters $g_1$, $g_2$.  The remaining entries $a_{in}$ are given by algebra, i.e., using $(c_i,c_6)=0$ for $i=3,4,5,6$.
This illustrates that our algorithm does not always give an injective map, indeed we may replace $\nu_1$ and $\nu_2$ by $\nu_1+c$ and $\nu_2-c$ for any constant $c$.  Also, although it is surjective \emph{locally} as $\nu_1$ and $\nu_2$ can be found by integration from $a_{13}$ and $a_{14}$, it is not \emph{globally} surjective.  For example, if $M = S^2$, $g_1 = g_2 = z$ and $a_{13}=-a_{14} = 1/z$, then
$\nu_1 = -\nu_2 =  \int(1/z)dz = \log z$ which is not globally defined, though $a_{15} = 0$ is.

However, we can modify our algorithm for this case as follows. Replace the final new parameters
$\nu_1$ and $\nu_2$ by $\wt\nu_1$, $\wt\nu_2$ with $a_{13} = \wt\nu_1$ and $a_{15} = \wt\nu_2 + \la\nu_3$, 
then we obtain
$$
A = \begin{pmatrix} 1 & 0 & \wt\nu_1 & (\wt\nu_2'-g_2'\wt\nu_1)/g_1'
										  & \wt\nu_2+\la\nu_3  & a_{16} \\
					0 & 1 & -g_1   & -g_2 &  -g_1 g_2 & a_{15} \\
					0 & 0 &   1    &   0  &   g_2     & a_{14} \\
					0 & 0 &   0    &   1  &   g_1     & a_{13} \\
					0 & 0 &   0    &   0  &    1      & 0 \\
					0 & 0 &   0    &   0  &    0      & 1
\end{pmatrix}.
$$
(which holds even if $g_2$ is constant)
where the remaining entries $a_{ij}$ are calculated by algebra, as usual.
The resulting harmonic maps are described by
Proposition \ref{prop:r=2}(i).

\medskip

(d) {\bf Type (1,2,2,1)}.
This has $r=3$ and is obtained from type $(2,2)$ by adding a border; it has two new initial parameters $\nu_1^1$, $\nu_2^1$ in the $\g_1$, i.e., block superdiagonal positions $a_{12}, a_{13}$, and two further parameters $\nu_3, \nu_4$ on the second block superdiagonal.  Carrying out our algorithm in the non-degenerate case when $g$ is non-constant replaces
$\nu_1^1$, $\nu_2^1$ by $\nu_1$, $\nu_2$ giving 
$$
A = \begin{pmatrix} 1 & (\nu_1)^{(1)} & (\nu_2)^{(1)} & \nu_1+\la \nu_3
				   & -\nu_2+\la \nu_4 & \zeta_0+\la\zeta_1 \\
0 & 1 & 0 & g &  0 & -g \nu_2^{(1)} + \nu_2 - \la \nu_4 \\
0 & 0 & 1 & 0 & -g &  g \nu_1^{(1)} - \nu_1 - \la \nu_3	\\
0 & 0 & 0 & 1 & 0 & -\nu_2^{(1)} \\
0 & 0 & 0 & 0 & 1 & -\nu_1^{(1)} \\
0 & 0 & 0 & 0 & 0 & 1 \end{pmatrix}.
$$

Here our final parameters $g,\nu_1, \nu_2, \nu_3, \nu_4$ are arbitrary meromorphic functions, and
all generalized derivatives are taken with respect to $g$.
The top-right entry $\zeta_0+\la\zeta_1$ is determined by algebra from
$(c_6,c_6)=0$, in fact, $\zeta_0 = \nu_1^{(1)}\nu_2 - \nu_2^{(1)}\nu_1$ and 
$\zeta_1 = -\nu_1^{(1)}\nu_4 - \nu_2^{(1)}\nu_3$. 
We now calculate the corresponding extended solution.
Write the $j$th column of $A$ as $c_j = c_j^0 + \la c_j^1$, so $c_j^1=0$ \ $(j=1,2,3)$; then, as in \S \ref{subsec:On-1},
$$
W = \spa{c_6} \oplus \la\spa\{c_6,c_5,c_4\} \oplus
\la^2\spa\{c_6,c_5,c_4,c_3,c_2\}  + \la^3\HH_+\,.
$$
This is the extended solution of a map into a Grassmannian if and only if $\nu_3 \equiv \nu_4 \equiv 0$; in that case we have an $S^1$-invariant extended solution:
$$
W = \delta_1 \oplus \la\delta_2 \oplus \la^2 \delta_3 + \la^3\HH_+
$$
where $0 = \delta_0 \subset \delta_1 \subset \delta_2 \subset \delta_3 \subset \delta_4 = \CC^6$ are the subbundles given by
 $\delta_1 = \spa\{c_6^0\}$, $\delta_2 = \spa\{c_6^0,c_5^0,c_4^0\}$ and 
$\delta_3 = \spa\{c_6^0,c_5^0,c_4^0,c_3^0,c_2^0\}$.
Note that $\delta_3$ is the polar of $\delta_1$ and $\delta_2$ is self-polar, i.e., maximally isotropic. 
As in \S \ref{subsec:On-1}, the corresponding
harmonic map $\varphi_0$ is $\psi_0 \oplus \psi_2$ where
$\psi_i = \delta_i ^{\perp} \cap \delta_{i+1}$, or its orthogonal complement 
$\psi_1 \oplus \psi_3$.  Since these are conjugates of each other, 
$\varphi_0$ is a harmonic map into $\O(6)/\U(3)$.

In the general case with $\nu_3$ or $\nu_4$ not necessarily zero,
we calculate the alternating factorization \eqref{fact} into unitons from \eqref{alt-fact} and \eqref{alpha} to be 
$\Phi = (\pi_{\al_1} + \la\pi_{\al_1^{\perp}}) 
(\pi_{\al_2} + \la\pi_{\al_2^{\perp}})(\pi_{\al_3} + \la\pi_{\al_3^{\perp}})$
where 
$\al_1 = \delta_2$,
$\al_2 = \spa\{c_6^0 + \pi_{\delta_2}^{\perp}c_6^1,\,
c_5^0 + \pi_{\delta_2}^{\perp}c_5^1,\, c_4^0 + \pi_{\delta_2}^{\perp}c_4^1\} \oplus \psi_2$ and
$\al_3 = \spa\{c_6^0 + \pi_{\delta_2}^{\perp}c_6^1\}$.
Note that $\al_1$ is maximally isotropic, and $\al_3$ is isotropic and is the polar of $\al_2$  (in the $S^1$-invariant case, $\al_2=\delta_3$ and
$\al_3 = \delta_1$).
The corresponding harmonic map $\varphi = \ii\Phi_{-1}$ is given by the product of the two maps into $\O(6)$: 
$\varphi = \ii(\pi_{\delta_2} - \pi_{\delta_2}^{\perp}).
(\pi_{\al_3 + \ov{\al_3}} - \pi_{\al_3 + \ov{\al_3}}^{\perp})$.
This example is related to \cite[Example 6.22]{unitons} (that example has a $\la^2$-term which can be removed by a suitable transformation of the data);
it provides extended solutions of harmonic maps into $\O(6)$ which do not lie in a Grassmannian but have $S^1$-invariant limits into $\O(6)/\U(3)$.

\medskip

(e) {\bf Type (1,1,2,1,1)}  This has $r=4$ and, like type (2,2,2) above, is obtained from $n=4$, type $(1,2,1)$ by adding a border.  However, due to the special nature of 
$\SO(4)$ as being double-covered by the product of $\SU(2)$ with itself, there is an easier way which involves \emph{first finding the new last column} of $A$ then filling in the top-right element and new first row by algebra (see \S \ref{subsec:add-border}); for the $S^1$-invariant case this is as follows, with
all generalized derivatives with respect to $g_1$:

Write the last column as $[1,\chi_1,\chi_2,\chi_3,\chi_4,\zeta]^{\TT}$.
{}From the extended solution equation \eqref{GrM} we have the following, assuming that $g_1$ is non-constant. 

(i) First, $\chi_2' = g_1 \chi_1'$.  Integrating by parts gives
$$
\chi_2 = g_1 \chi_1 - \int\! g_1' \chi_1\,.
$$
Replace $\chi_1$ by a new parameter $h_1$ and set $\chi_1 = h_1^{(1)} := h_1'/g_1'$.  Then
$\chi_2 = g_1 h_1^{(1)} - h_1$.

(ii) Next, $\chi_3' = g_2 \chi_1'$ so that
$\chi_4' =  -g_1 g_2 \chi_1' = -g_1 \chi_3'$.
{}From this equation we have, in a similar way to (i),
$$
\chi_4 = -g_1 \chi_3 + \int\! g_1' \chi_3\,.
$$
Replace $\chi_3$ by a new parameter $h_2$ and set $\chi_3 = h_2^{(1)} := h_2'/g_1'$.  Then
$\chi_4 = -g_1 h_2^{(1)} + h_2$.

The remaining entries $a_{1j}$ are found by algebra, i.e., using 
$(c_j,c_n)=0$ for $j=2,\ldots, n$.
Thus we obtain, with generalized derivatives taken with respect to $g_1$,
\be{11211}
A = \begin{pmatrix}
1 & -h_1^{(1)} & h_1 & g_2 h_1^{(1)}-h_2^{(1)} & g_2 h_1 - h_2 
									& h_1 h_2^{(1)} - h_2 h_1^{(1)} \\
0 & 1       & -g_1   & -g_2   & -g_1 g_2 & h_2 - g_1 h_2^{(1)} \\
0 & 0       & 1      & 0      & g_2      & h_2^{(1)} \\
0 & 0       & 0      & 1      & g_1      & -h_1 + g_1 h_1^{(1)} \\
0 & 0       & 0      & 0      & 1        & h_1^{(1)} \\
0 & 0       & 0      & 0      & 0        & 1 
\end{pmatrix}.
\ee
Here $g_1,h_1,h_2$ are arbitrary meromorphic functions.  If $h_1^{(1)}$ is non-constant, then $g_2 = h_2^{(2)}/h_1^{(2)}$.
Note how this departs from our usual algorithm by replacing a parameter in the middle $4 \times 4$ matrix $\wt A$, in this case $g_2$ by $h_2$.  Note that the parameters $g_1,h_1,h_2$ can be read off from the matrix $A$ as entries, or combinations of entries.
Note also that the middle four entries of the last column give the standard formula for null curves in $\C^4$,
see \S \ref{subsec:null-R4}.

\begin{proposition} For \ $m \leq 3$, any harmonic map of finite uniton number  $M \to \O(2m)/\U(m)$ has an $S^1$-invariant associated extended solution.
\end{proposition}

\begin{proof}
As in \S \ref{subsec:On-1}, $\varphi$ has a symmetric extended solution
$\Phi= [A\ga_{\xi}]$ with $r$ odd.
By Remark \ref{rem:parity}, if $\Phi$ is not $S^1$-invariant then $A$ must contain a term in $\la^2$.  By Lemma \ref{lem:top-right-deg} this means that, either $r=3$ with $t_1 >1$, or $r \geq 5$.  Given that $\sum_{i=1}^r{t_i} = 2m$, neither of these is possible with $m \leq 3$.
\end{proof}

That this result is sharp is shown by the following example which is a particular case of \cite[Example 6.26]{unitons}.  In that paper, reality conditions had to be solved: this was only done for $m \geq 5$;  by using our approach, the reality conditions in that example are automatic and give us an example for $m=4$. Explicitly, take
$\xi$ of type $(2,2,2,2)$.  By our method we may construct a solution $A: M \to \A_3^{\rr}$ in the form $A = A_0 + \la^2 A_2$ where the penultimate entry of the top row of $A_2$ is a freely chosen parameter $\nu$.  Completing the matrix $A$ by algebra and setting $\Phi=[A\ga_{\xi}]$ gives an extended solution which is $S^1$-invariant if and only if $\nu\equiv 0$.

\section{Null curves, extended solutions and the Weierstrass representation} \label{sec:null-isotropic}

By a \emph{(generalized) minimal surface in $\R^n$} we mean a
non-constant weakly conformal map from a Riemann surface $M$ to $\R^n$ whose image is minimal away from branch points, equivalently, a weakly conformal harmonic map.  Such a map is, on a simply connected domain, the real part of a \emph{null holomorphic curve} by which is meant (somewhat confusingly) \emph{a holomorphic map 
$\chi:M \to \C^n$ with $(\chi',\chi') = 0$ and $\chi'$ not identically zero}.  We extend this definition to \emph{null meromorphic curve}: note that for such a curve,
 $[\chi']:M \to Q_{n-2}$ is a well-defined holomorphic map to the complex quadric and gives the Gauss map of the minimal surface.  The usual Weierstrass representation parametrizes all such $\chi'$ so that $\chi$ is given by an \emph{integral} with real part the minimal surface.  In contrast, in the \emph{Weierstrass representation in free form}, the null curve itself is parametrized and no integral is necessary.  We see how this is related to our work.

\subsection{Null curves in $\C^3$ and extended solutions} \label{subsec:null-R3}

Let $M$ be a Riemann surface.     In Theorem \ref{th:11-appl}, we related Calabi's construction of harmonic maps into spheres with polynomial extended solutions of harmonic maps into $\O(n)$ ($n$ odd) of type $(1,1,\ldots, 1)$.  In the case $n=5$ we can add one further bijection: that with null meromorphic curves, showing how the Weierstrass representation in free form appears naturally from polynomial extended solutions for harmonic maps into $\O(5)$ of type $(1,1,1,1,1)$ and so of the maximum possible uniton number $4$; the corresponding canonical element is $\xi_0 = \ii\diag(4,3,2,1,0)$.  In part (iii), the generalized derivative $\nu^{(2)}$ is taken with respect to $g$.

\begin{theorem} \label{th:Weier3}
The following sets are in one-to-one correspondence$:$
\begin{itemize}
\item[(i)] null meromorphic curves\/ $\chi:M \to \C^3$ with $[\chi']:M \to Q_1$ non-constant$;$

\item[(ii)] non-degenerate $S^1$-invariant extended solutions $\Phi:M \to \Omega_4\U(5)^{\rr}$ of type $(1,1,1,1,1);$

\item[(ii)$'$] non-degenerate solutions $A:M \to (\A_{\xi_0}^{\rr})_0$ to the extended solution equation \eqref{GrM}$;$

\item[(iii)] pairs of meromorphic functions $(g,\nu)$ on $M$ with $g$ and $\nu^{(2)}$ non-constant$;$

\item[(iv)] full totally isotropic holomorphic maps $f:M \to \CP^4;$

\item[(v)]  full harmonic maps $\varphi:M \to \RP^4;$

\item[(vi)] antipodal pairs $\pm\wt\varphi:M \to S^4$ of full
harmonic maps.
\end{itemize}

The bijection from (ii)$'$ to (ii) is given by $\Phi=[A\ga_{\xi_0}]$ as in
Proposition \ref{prop:BuGu-On}.

The bijection from (ii)$'$ to (i) is given by 
$\chi = (a_{45},a_{35},a_{25})$.
That from (ii)$'$ to (iii) is given by
 \be{A-g-nu}
g= a_{34} \quad \text{and} \quad \nu = a_{14},
\ee 
 and that from (ii)$'$ to  (iv) is given by taking the last column: $f = [c_5]$ as in
Theorem \ref{th:11-appl}.  
\end{theorem}

\begin{proof}
Given $\chi = (\chi_1,\chi_2,\chi_3)$ as in (i), note first that
$\chi_1'$ is not identically zero; otherwise
since $\chi$ is a null curve, $\chi_1'\chi_3' = -\half \chi_2'$ so that $\chi_2' \equiv 0$ which implies that $[\chi']$ is constant.  There is a unique solution $A:M \to \O(5,\C)$ of type $(1,1,1,1,1)$ to the extended solution equation \eqref{GrM1k-1} with the middle of the last column given by $(a_{45},a_{35},a_{25}) = (\chi_1,\chi_2,\chi_3)$, namely,
\be{11111-chi}
A = \begin{pmatrix}
1 & -\chi_1     &g\chi_1-\chi_2   &\half g^2\chi_1-g\chi_2-\chi_3 &-\chi_1\chi_3 - \half\chi_2^{\;2} \\
0 & 1         & -g      & -\half g^2 & \chi_3 \\
0 & 0         & 1       & g          & \chi_2 \\
0 & 0         & 0       & 1          & \chi_1  \\
0 & 0         & 0       & 0          & 1 \end{pmatrix}
\quad \text{where } g = \chi_2'/\chi_1'.  
\ee
Indeed, all but the first row of $A$ is found by differentiating
four times the middle of the last column; the remaining entries $a_{1j}$ are filled in by algebra, i.e., using $(c_j,c_5)=0$ for $j=2,3,4,5$.  
Thus $\chi \mapsto A$ gives a bijection from set (i) to set (ii)$'$ with inverse $\chi = (a_{45},a_{35},a_{25})$.

\smallskip

Given $A$ as in (ii)$'$, define $(g,\nu)$ by \eqref{A-g-nu}.
 {}From the extended solution equation \eqref{GrM1k-1}, with generalized derivatives with respect to
$g$, $a_{13}= \nu^{(1)}$ and $a_{12}= -\nu^{(2)}$; then $A$ is given by \eqref{11111}. By non-degeneracy of $A$, $g$ and $\nu^{(2)}$ are non-constant.
The assignment $A \to (g,\nu)$ gives a bijection between sets (ii)$'$ and (iii) with inverse given by \eqref{11111}.
\end{proof}

Composing the above bijections we deduce the Weierstrass representation in free form of null meromorphic curves:
\begin{corollary}
There is a bijection between the following sets$:$
\begin{enumerate}
\item[(i)] the set of pairs of meromorphic functions $(g,\nu)$ on $M$ with $g$ and $\nu^{(2)}$ non-constant,
\item[(ii)] the set of null meromorphic curves\/ $\chi:M \to \C^3$ with $[\chi']:M \to Q_1$ non-constant,
\end{enumerate}
given by
\be{Weier3}
\chi = (\nu^{(2)},\, -\nu^{(1)} + g \nu^{(2)},\, -\nu + g\nu^{(1)} - \half g^2 \nu^{(2)}).
\ee
\end{corollary}

Recall that minimal surfaces in $\R^3$ appear as the real part of such curves $\chi$.
The representation \eqref{Weier3} seems to have been first given by K.~Weierstrass \cite{weierstrass}; explanations are given by N.J.~Hitchin
\cite{hitchin} and A.~Small \cite{small-R3}. 
The new feature in our work is the correspondence with extended solutions for harmonic maps, 
specifically, \emph{the free Weierstrass data} $(g,\nu)$ of $\chi$ is given simply by the two entries \eqref{A-g-nu} of the matrix $A$ associated to $\chi$ by \eqref{11111-chi}, and this matrix defines an extended solution $\Phi=[A\ga_{\xi_0}]$ for a harmonic map
$M \to \O(5)$.

\subsection{Null curves in $\C^4$ and extended solutions} \label{subsec:null-R4}

Theorem \ref{th:Weier3} has an analogue in $\C^4$ as follows.  For a null curve
$\chi=(\chi_1,\chi_2,\chi_3,\chi_4):M \to \C^4$, by definition, $\chi'$ 
is not identically zero, so by permuting coordinates if necessary, we can assume that $\chi_1$ is non-constant.  Then
we can set $g_1 = \chi_2'/\chi_1'$ and $g_2= \chi_3'/\chi_1'$
so that $[\chi'] = [1,g_1,g_2,-g_1g_2]$ and $[\chi']$ is non-constant if and only if at least one of the \emph{Gauss maps} $g_1$ or $g_2$ is non-constant; again, after permuting coordinates, if necessary, we can assume that $g_1$ is non-constant.  By \emph{$A$ non-degenerate} we shall now mean that $a_{i,i+1}$ is non-constant for $i \neq 3$.
The extended solutions in (ii) below are polynomial extended solutions for harmonic maps into $\O(6)$, and, as in the $\C^3$ case, are of type $(1,1,2,1,1)$, and so of the maximum possible uniton number, $4$; the corresponding canonical element is $\xi=\ii\diag(4,3,2,2,1,0)$.

\begin{theorem} \label{th:Weier4}
The following sets are in one-to-one correspondence$:$
\begin{itemize}
\item[(i)] null meromorphic curves\/ $\chi:M \to \C^4$ with $\chi_1$ and $g_1 := \chi_2'/\chi_1'$ non-constant$;$

\item[(ii)] non-degenerate $S^1$-invariant extended solutions $\Phi:M \to \Omega_4\U(6)^{\rr}$ of type $(1,1,2,1,1);$

\item[(ii)$'$] non-degenerate solutions $A:M \to (\A_{\xi}^{\rr})_0$ to the extended solution equation \eqref{GrM}$;$

\item[(iii)] triples of meromorphic functions $(g_1,h_1,h_2)$ on $M$ with $g_1$ and
$h_1^{(1)} := h_1'/g_1'$ non-constant.
\end{itemize}

The bijection from (ii)$'$ to (ii) is given by $\Phi=[A\ga_{\xi}]$ as in
Proposition \ref{prop:BuGu-On}.

The bijection from (ii)$'$ to (i) is given by
$\chi = (a_{56},a_{46},a_{36},a_{26})$.
 
The bijection from (ii)$'$ to (iii) is given by
\be{A-e-h}
g_1 = a_{45}, \quad h_1 = a_{13}, \quad h_2 = a_{13}a_{35} - a_{15}.
\ee
\end{theorem}

\begin{proof}
Given $\chi$ in set (i) there is a unique $A$ in set (ii)$'$ which
satisfies $\chi = (a_{56},a_{46},a_{36},a_{26})$, namely,
\be{11211a}
A = \begin{pmatrix}
1 & a_{12} & a_{13}  & a_{14} & a_{15} 	 & a_{16}  \\
0 & 1       & -g_1   & -g_2   & -g_1g_2  & \chi_4 \\
0 & 0       & 1      & 0      & g_2      & \chi_3 \\
0 & 0       & 0      & 1      & g_1      & \chi_2 \\
0 & 0       & 0      & 0      & 1        & \chi_1 \\
0 & 0       & 0      & 0      & 0        & 1 
\end{pmatrix}.
\ee
Here $g_1 = \chi_2'/\chi_1'$ and $g_2 = \chi_3'/\chi_1'$; the remaining entries $a_{1j}$ can be  found by algebra, i.e., using $(c_j,c_6)=0$ for $j=2,3,4,5,6$.
Thus $\chi \mapsto A$ gives a bijection from set (i) to set (ii)$'$ with inverse $\chi = (a_{56},a_{46},a_{36},a_{26})$.
 
Given $(g_1,h_1,h_2)$ in set (iii), set $A$ equal to \eqref{11211} where $g_2 = h_2^{(2)}/h_1^{(2)}$.  It is easily checked that this is the inverse of the map \eqref{A-e-h}.
\end{proof}

\begin{corollary}
There is a bijection between the following sets$:$
\begin{enumerate}
\item[(i)] the set of triples of meromorphic functions $(g_1,h_1,h_2)$ on $M$ with $g_1$ and $h_1^{(1)}$ non-constant,
 \item[(ii)] the set of null meromorphic curves\/ $\chi:M \to \C^4$ with $\chi_1$ and $g_1 := \chi_2'/\chi_1'$ non-constant,
\end{enumerate}
given by
\be{Weier4}
\chi = (h_1^{(1)},\, -h_1 + g_1 h_1^{(1)},\, h_2^{(1)},\, h_2 - g_1 h_2^{(1)}).
\ee
\end{corollary}
Again minimal surfaces in $\R^4$ appear as the real part of such $\chi$.
This seems to have been first given by M.~de~Montcheuil
\cite{montcheuil}, see also L.~Eisenhart \cite{eisenhart};  explanations are given by A.~Small \cite{small-R4} and W.T.~Shaw \cite{shaw}. 
As before, the free Weierstrass data $(g_1,h_1,h_2)$ of $\chi$ are given very simply by \eqref{A-e-h} from the entries of the matrix $A$
associated to $\chi$ by \eqref{11211a}, and this matrix defines an extended solution $\Phi=[A\ga_{\xi}]$ for a harmonic map
$M \to \O(6)$.

\end{document}